\documentclass[10pt,reqno]{elsarticle}
\usepackage{amsmath,amssymb,amsthm}
\usepackage[latin1]{inputenc}
\usepackage{version,tabularx,multicol}
\usepackage{graphicx,float, url}
\usepackage{amsfonts}
\usepackage{soul}

\usepackage{graphicx}
\usepackage{epsfig}
\usepackage{subfigure}
\usepackage{float}
\usepackage{flafter}
\newcommand{\ind}{{\bf 1}}

\usepackage{url}

\theoremstyle{plain}
\newtheorem{theorem}{Theorem}[section]
\newtheorem{proposition}[theorem]{Proposition}
\newtheorem{lemma}[theorem]{Lemma}

\newtheorem{corollary}[theorem]{Corollary}

\newtheorem{definition}[theorem]{Definition}
\newtheorem*{assumption}{Assumption A}

\theoremstyle{definition}

\theoremstyle{remark}
\newtheorem{remark}[theorem]{Remark}

\newcommand{\N}{\ensuremath{\mathbb{N}}}

\newcommand{\R}{\ensuremath{\mathbb{R}}}
\renewcommand{\P}{\ensuremath{\mathbb{P}}}
\newcommand{\E}{\ensuremath{\mathbb{E}}}

\newcommand{\var}{\ensuremath{{\rm{var}}}}

\newcommand{\be}{\begin{equation}}
\newcommand{\ee}{\end{equation}}

\vspace{0.3cm}

\begin{document}

\begin{frontmatter}

\title{An individual-based model for the Lenski experiment,\\ and the deceleration of the relative fitness}

\author[tuaddress]{Adri\'an Gonz\'alez Casanova}
\ead{adriangcs@hotmail.com}

\author[tuaddress]{Noemi Kurt}
\ead{kurt@math.tu-berlin.de}

\author[fraddress]{Anton Wakolbinger}
\ead{wakolbin@math.uni-frankfurt.de}

\author[upaddress]{Linglong Yuan}
\ead{yuanlinglongcn@gmail.com}

\address[tuaddress]{TU Berlin, Institut f\"ur Mathematik, Stra\ss e des 17. Juni 136, 10623 Berlin, Germany}

\address[fraddress]{Mathematical Institute, Goethe-University Frankfurt, Box 111932, 60054 Frankfurt am Main, Germany}

\address[upaddress]{Department of Mathematics, Uppsala University, SE-75106 Uppsala, Sweden}

%
%
%
%
\date{\today}

\begin{abstract}
The Lenski experiment investigates the long-term evolution of bacterial populations. In this paper we present an individual-based probabilistic model that captures essential features of the experimental design, and whose mechanism does not include epistasis in the continuous-time (intraday) part of the model, but leads to an epistatic effect in the discrete-time (interday) part. We prove that under some assumptions excluding clonal interference, the rescaled relative fitness process converges in the large population limit to a power law function, similar to the one obtained by Wiser et al. (2013), there attributed to effects of clonal interference and epistasis. 
\end{abstract}


\end{frontmatter}

\section{Introduction}
The \textit{Lenski experiment} (see \cite{Lenski1, LT, lenski} for a detailed description) is a cornerstone in experimental evolution. It investigates the long-term evolution of 12 initially identical populations of the bacterium E. coli in identical environments. 
One of the basic concepts of the Lenski experiment is that of the {\it daily cycles}. Every day starts by sampling approximately $5\cdot  10^6$ cells from the bacteria available in the medium that was used the day before. This sample is 
propagated in a minimal glucose medium. The bacteria then reproduce (by binary splitting) with an exponential population growth. The reproduction continues until the medium is deployed,
i.e., when there is no more glucose available. Then the reproduction stops and a phase of starvation starts. This phase lasts until the beginning of the next day, when the new sample is transferred
to fresh medium. Around $5\cdot  10^8$ cells are present at the end of each day.

Up to now the experiment has been going on for more than {60'000 generations (or 9000} days, see \cite{lenski}). One important feature is that samples of  ancestral populations
were stored, which afterwards could be made to reproduce under competition with later generations in order to experimentally determine the fitness of an evolved strain relative to the founder ancestor of the population by comparing their growth rates in the following manner \cite{Lenski1}: A population of size $A_0$ of the unevolved strain and a population of size $B_0$ of the evolved strain perform a direct competition in the minimal glucose medium. The respective population sizes at the end of the day, that is, after the glucose is consumed, are denoted by $A_1$ and $B_1.$ The (empirical) \emph{relative fitness} $F(B|A)$ {of strain $B$ with respect to strain $A$} is then given by the ratio of the exponential growth rates, calculated as 

\begin{eqnarray}\label{eq:empirical_fitness}
F(B|A)&=&\frac{\log(B_1/B_0)}{ \log(A_1/A_0)}.
\end{eqnarray}

Considerable changes of the relative fitness have been observed in the more than {25} years of the experiment (\cite{LT, Barrick, WRL}). As expected, the relative fitness of the population increases over time, but one of the features that have been observed is a pronounced deceleration in the increase of the relative fitness, see Figure 2 in \cite{WRL}. In particular it has been observed that it increases sublinearly over time.  Several questions have arisen in this context (\cite{Barrick, WRL}):  How can the change of relative fitness be explained or approximated? Which factors account for the deceleration in the increase of the relative fitness? 

In \cite{Barrick}, the authors perform an analysis on the change of the relative fitness for the first 20'000 generations of the experiment, and of the mutations that go to fixation
during the same period. They conjecture that effects of dependence between mutations, like \emph{clonal interference} and \emph{epistasis}, contribute crucially to the deceleration of the gain of relative fitness. 

 In \cite{WRL}, the authors analyse the change of the relative fitness for the first 50'000 generations of the experiment, and fit {the observations} to a power law function. They also conjecture that clonal interference and epistasis contribute crucially to the quantitative behavior of relative fitness, and support this conjecture by sketching a mathematical model which predicts a power law function for the relative fitness.
 
In this paper, we propose a basic mathematical model for a population that captures  essential features of the Lenski experiment, in particular the daily cycles. It models an asexually reproducing population whose growth in each cycle is stopped after a certain time, and a new cycle is started with a sample of the original population. We include (beneficial) mutations into the model by assuming that an individual may mutate with a certain (small) probability and draws a certain (small) reproductory benefit from the mutation, which results in an increase of the reproduction rate during the cycle. We then calculate the probability of fixation of a beneficial mutation, and its time to fixation. Using this, we can prove that under some conditions on the parameters of mutation and selection, with  high probability there will be no clonal interference in the population, which means in our situation that, with high probability, beneficial mutations only arrive when the population is homogeneous (in the sense that 
all its individuals have the same {reproduction rate}). This result implies that we are essentially dealing with a model of adaptive evolution, which allows a thorough mathematical analysis. In particular, using convergence results for Markov chains in the spirit of \cite{kurtz}, we are able to prove that the relative fitness of the population, on a suitable timescale in terms of the population size, converges locally uniformly to a deterministic curve (see Figure \ref{fig:parabola}). 

In this way we arrive at an explanation of a power law behavior (with a deceleration in the increase) of the relative fitness. This explanation is in terms of the experiment's design (which makes the {\em generation time} dependent of the fitness level), and does not invoke clonal interference, nor a direct epistatic effect of the beneficial mutations (see Sec.~\ref{sect:epistasis}).

More specifically, in our model every beneficial mutation which is succesful in the sense that it goes to fixation, will increases the individual reproduction rate by the same amount  ($\rho$, say), irrespective of the current value $r$ of the individual reproduction rate. In this sense the model is ``non-epistatic''.
However, there will be an indirect epistatic effect caused by the design of the experiment:  since the amount of glucose, which the bacteria get for their population growth, remains the same from day to day, a population with a high individual reproduction rate will consume this amount more quickly than a population with a low individual reproduction rate. In other words, the daily duration of the experiment (that is the time $t = t_i$ during which the population grows at day $i$) will depend on the current level $r=r_i$ of the individual reproduction rate, and will  become shorter as $r$ increases.
Indeed, the ratio of the two expected growth factors in one day is
$\exp((r+\rho) t) / \exp (rt)  = \exp (\rho t)$.
Even though $\rho$ does not depend on $r$ by our assumption, this ratio does depend on $r$, because, as stated above, the duration $t = t_i$ of the daily cycles becomes smaller as $r$ increases. We are well aware that clonal interference as well as direct epistatic effects will also be at work in the Lenski experiment, and should be modelled. On the other hand, our results might help to separate these effects from an indirect epistatic effect caused by a shortening of the daily cycles as the generations proceed, which would go along with a quicker consumption of the daily nutrition as fitness increases.
%
 
In the remainder of this introductory section we discuss our mathematical approach and main results, and put our methods into the context of related work. The formal statement of the model and the main results will be given in Section \ref{sec:modelandresults}, and the proofs in Section \ref{proofsec}. The most intricate proof is that of Theorem \ref{thm:fixation} which relies on a coupling of the daily sampling scheme with near-critical Galton-Watson processes that is successful over a sufficiently long time period. Some tools from the theory of branching processes (Yule and Galton-Watson processes) are presented in the Appendix.
\subsection{A neutral model for the daily cycles}\label{neutralmodel}
We build our model on few basic assumptions: Every individual reproduces independently by binary splitting at a given rate until the end of a growth cycle, which corresponds to one {\em day} (in the sense of \cite{lenski}). 
Our daily cycle model is determined by specifying the reproduction rate of each individual, and a stopping rule to end the growth of the population. To illustrate this we assume for the moment a \emph{neutral situation}, i.e. all individuals have the same reproduction rate. The experiment is laid out such that the total number of bacteria at the end of one day is roughly the same for every day. This suggests the following mathematical assumptions: Each day starts with a population of $N$ individuals. These individuals reproduce by binary splitting at some fixed rate $r$ until the maximum capacity is reached. We assume that this happens (and that the ``Lenski day'' is over)  as soon as the total number of cells in the medium is close to $\gamma N$ for some constant $\gamma>1$ (a precise definition and a discussion of the corresponding stopping rules will be given in Section \ref{sect:daily}). The description of the experiment suggests to think of $N=5\cdot 10^6$, and $\gamma \approx 100$, since at the end of 
each day, one gets around $5\cdot 10^8$ bacteria, see supplementary material of \cite{Lenski1}. The subsequent day is started  by sampling $N$ individuals from the approximately $\gamma N$ total amount available, and the procedure is repeated. 

This setting induces a genealogical process, which we study on the evolutionary time scale, that is with one unit of time corresponding to $N=5\cdot 10^6$ days. On this time scale, the genealogical process turns out to be approximately a constant time change of the Kingman coalescent, where the constant is $c_\gamma: = 2(1-\frac{1}{\gamma})$. In this sense, $N/c_\gamma $ plays the role of an \emph{effective population size}. With the stated numbers, this is much larger than the number ($\approx 9000$) of ``Lenski days''   that have passed so far. In other words, in the neutral model so far only a small fraction of one unit of the evolutionary timescale has passed. Still, this model provides a good basis to introduce mutation and selection. In fact, we will see that the design of the experiment (via the stopping rule that defines the end of each day) affects the selective advantage provided by a beneficial mutation and in this way has an influence that 
goes well beyond the determination of the effective population size in the neutral model.  

Our genealogical model arises naturally from the daily cycle setting, see Figure \ref{fig:Lenski1}. Schweinsberg \cite{Schweinsberg2003107} obtained a Cannings dynamics by sampling generation-wise $N$ individuals from a supercritical Galton-Watson forest, and analysed the arising coalescents as $N \to \infty$. Our model is similar in spirit, with the binary splitting leading to Yule processes. We will introduce the additional feature that some individuals reproduce at a faster rate;  in this sense Schweinsberg's sampling approach to neutral coalescents is naturally extended to a case with selection. 

\begin{figure}[ht]
    \centering
     \includegraphics[height=0.5\textwidth]{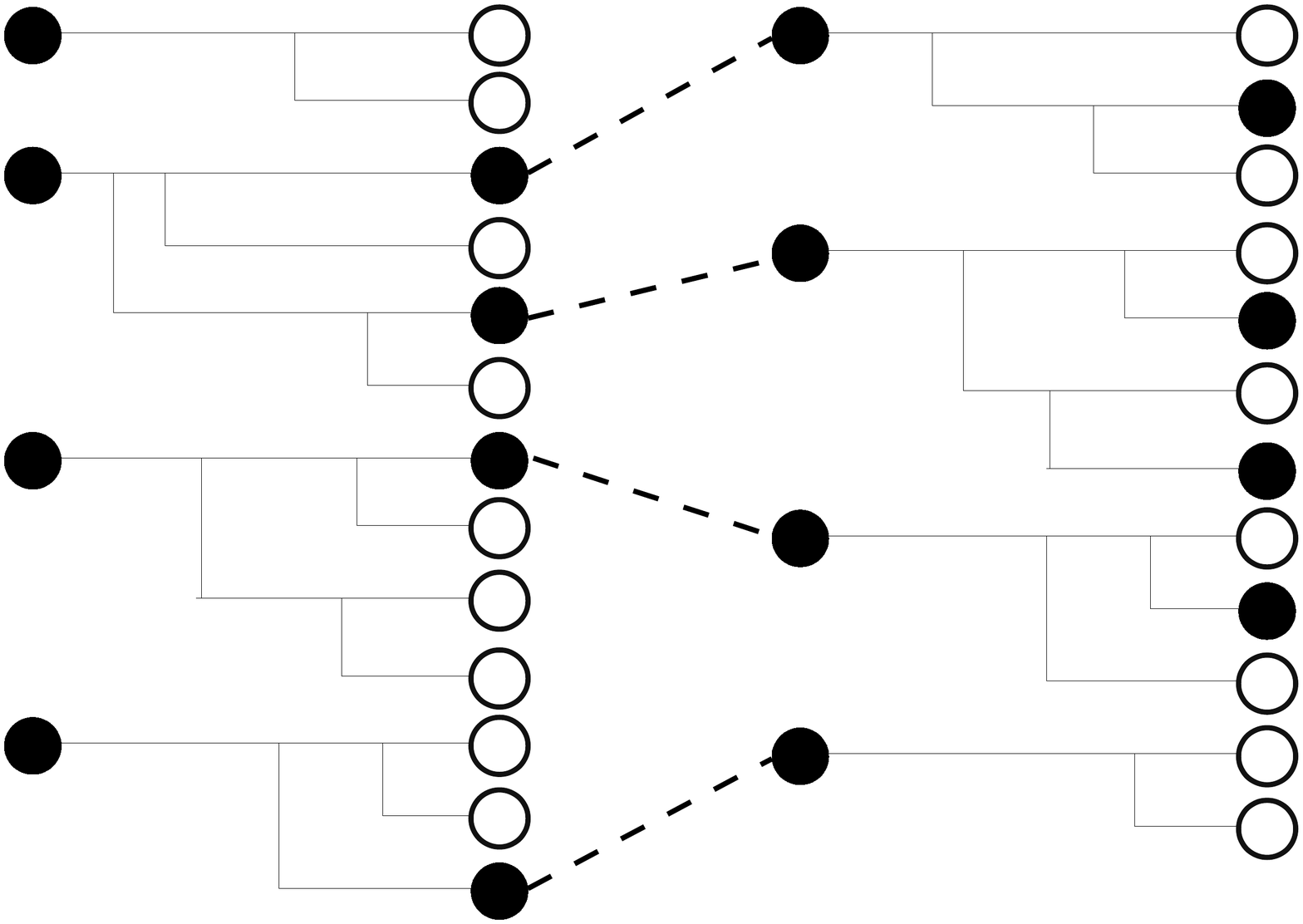}  
      \caption{\small Two of the daily cycles (or ``days''), with $N=4$ and $\gamma=3$. The $N$-sample at the end of day 1 constitutes the parental population at the beginning of day 2. }
      \label{fig:Lenski1}
\end{figure}

\subsection{Mutants versus standing population}
Next we consider a modification of the previous model, supposing that at a certain day a fraction of the population reproduces at rate $r$, while  the complementary fraction (founded by some beneficial mutant in the past) reproduces faster, say at rate $r+\varrho_N$, with $\varrho_N>0.$ Our assumptions will be that the {increment of the reproduction rate}  $\varrho_N$ is small, but not too small, more precisely we will assume that $\varrho_N\sim N^{-b}$ for some $0<b<1/2$ ($\sim$ denoting asymptotic equivalence, i.e. the convergence of the ratio to 1 as $N\to\infty).$ We assume that the reproduction rate is heritable. Based on the observation that with the stopping procedure indicated above a ``Lenski day'' lasts approximately $\frac{\log \gamma}{r}$ units of time of the Yule process, we will prove in Proposition \ref{prop:selective_advantage} that the expected number of offspring at the beginning of the next day of an individual with reproduction rate $r+\varrho_N$ is increased for large $N$ by 
approximately $\varrho_N\frac{\log \gamma}{r}$ compared to an individual with {reproduction rate} $r$. In this sense  the \emph{effective selective advantage} of a beneficial mutation is approximately $\varrho_N \frac{\log \gamma}{r}$.

Let us emphasize that here one obtains a dependence on the reproduction rate $r$ of the standing population due to the relation between $r$ and the ``length of a day'', 
i.e. the time span it takes the total population to reach the maximum capacity.  The implication of this result is that the selective advantage provided by reproducing $\varrho_N$ units faster
is comparatively large if the standing population is not well adapted and thus reproduces at a low rate, and is comparatively small if the population is well adapted in the sense that it already reproduces fast.

\subsection{Genetic and adaptive evolution}
In order to study the genetic and adaptive evolution of a population under the conditions of the Lenski experiment, we consider a model with {\it moderately strong selection -- weak mutation} and constant additive fitness effect of the mutations. We assume that the population reproduces in daily cycles as described above, and that at each day with probability $\mu_N$ a beneficial mutation occurs within the ancestral population of that day, where $\mu_N \to 0$ as $N\to \infty$. Following the ansatz described above, we assume that an individual affected by such a beneficial mutation increases its reproduction rate and that of its offspring by $\varrho_N$. Some of these mutations will go to fixation (in which case they will be called ``successful''), while the others are lost from the population. Calculating the probability of fixation of a beneficial mutation is a classical problem, studied already at the beginning of the last century by Haldane in the Wright-Fisher model. These questions still have a major 
interest in modern times, and have recently been studied in different contexts (see for example \cite{Lambert:2006} or \cite{Parsons}).

Assume now that the initial population on day $i$ consists of $N-1$ individuals that reproduce at rate $r$ and one mutant that reproduces at rate $r+\varrho_N$. We will see in Theorem \ref{thm:fixation} that the probability of fixation of such a mutant is asymptotically
\begin{equation}\label{profix}
\frac{\rho_N \log \gamma}{r} \frac {\gamma}{\gamma-1}
\end{equation}
as $N\to \infty$.
A crucial role in the proof of our result is played by an intricate approximation of the number of the mutants' descendants by near-critical Galton-Watson process, as long as their number is relatively small compared to the total population. 

In Proposition \ref{noclonalint}, we prove that in a certain regime of the model parameters, namely if $\varrho_N\sim N^{-b}, \mu_N\sim N^{-a},$ with $b\in (0,1/2)$ and $a>3b,$ the time it 
takes for a mutation to go to fixation or extinction is with high probability shorter than the time between two mutation events which is of order $\mu_N^{-1}$. 
This result allows us to exclude clonal interference on the time scale $ \lfloor t \varrho_N^{-2} \mu_N^{-1} \rfloor$, and to approximate the reproduction rate process of our original model 
by a simple Markov chain which can be interpreted as an idealized process where successful mutations fixate immediately on the scale of their arrival rate, and unsuccessful ones are neglected.

In this respect, the analysis presented in this paper can be seen in the framework of the theory of stochastic adaptive dynamics, as studied by Champagnat, M\'el\'eard and others, see \cite{Champagnat06, De13chemostatmodel} and references therein.  Let us emphasize, however, that we prove the validity of our approximation by taking {\em simultaneous limits} of the population size $N\to\infty,$ the rate of mutation $\mu_N\to 0,$ and the {increment of the reproduction rate} $\varrho_N\to 0,$ which requires some care, and is carried out by taking the specifics of our model into account.

{\subsection{Deterministic approximation on longer time scales}

The calculation of the fixation probability in Theorem \ref{thm:fixation} and the exclusion of clonal interference in Proposition \ref{noclonalint}, as well as the resulting Markov chain approximation of the reproduction rate process are the key steps in the analysis of the long-term behaviour of the population in the Lenski experiment.  This allows to derive the process counting the number of eventually successful beneficial mutations until a certain day, and the process of the relative fitness of the evolved population compared to the initial fitness.

It turns out, as we prove in Theorem \ref{thm:mutations} that for large $N$, on the time scale $ \lfloor t \varrho_N^{-1} \mu_N^{-1} \rfloor$, the number of successful mutations is approximately a Poisson process with constant rate $\frac{\gamma \log{\gamma}}{(\gamma-1)r_0},$ if the observation of the population starts on some day where the reproduction rate is constant and equal to $r_0>0.$

In order to define the fitness of an evolved strain relative to the unevolved one, we assume that the unevolved population, taken from the first day of the experiment, is homogeneous and evolves at rate~$r_0.$ 
In view of \eqref{eq:empirical_fitness} we define the fitness of the population at the beginning of day $i$ with respect to that at the beginning of day $0$ as
\begin{equation}\label{eq2}
F_i^{(N)} := \frac
{\log{\frac 1N}\sum_{j=1}^N e^{R_{i,j}u}}{\log e^{r_0u}} 
\end{equation}
where $R_{i,j}, j=1,\ldots, N$ are the reproduction rates of the individuals present at the beginning of day $i$, and $u$ is a given time for which the two populations are allowed to grow together. (This time may also depend on $i$, which does not affect our results.) For brevity we call $F_i$ the {\em relative fitness at day~$i$}.

We prove in Theorem \ref{thm:adapt} that, under the assumptions described above and specified in Sec. 2, the sequence of time-rescaled processes $(F_{\lfloor t \varrho_N^{-2} \mu_N^{-1}\rfloor})_{t\geq 0} $ converges locally uniformly as $N\to\infty$ to the parabola
\begin{equation}\label{parabola}
f(t)=\sqrt{1+\frac{2\gamma\log{(\gamma)}}{(\gamma-1)r^2_0}t}, \quad t\geq 0.
\end{equation}

Hence our model, which should be regarded as idealized and basic, still succeeds to describe the observed sublinear increase of relative fitness quite well on a qualitative level, even without incorporating the effects of clonal interference or epistasis.  

\begin{figure}[ht]
    \centering
     \includegraphics[scale=1]{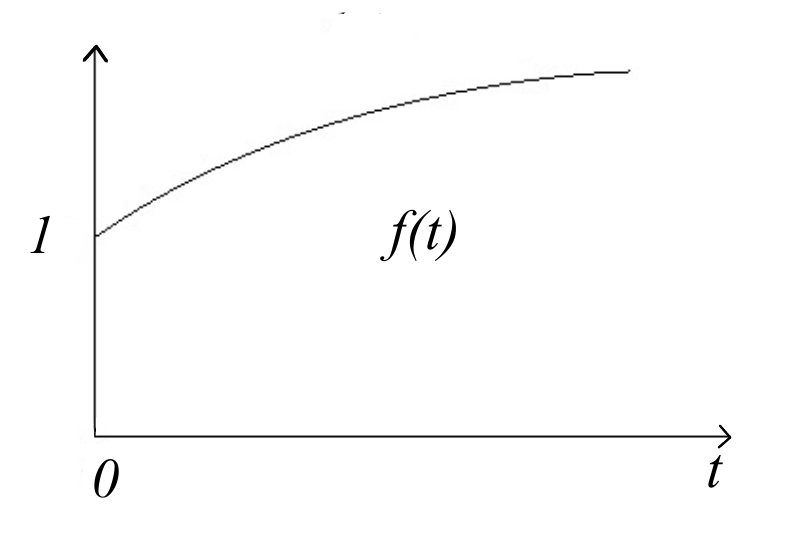} 
     \vspace{-0.5cm} 
      \caption{\small The limiting relative fitness curve for $N\to\infty,$ if time is rescaled by $\lfloor t \varrho_N^{-2} \mu_N^{-1}\rfloor, t\geq 0$. The curve $f(t)$ is given by \eqref{parabola}.}
      \label{fig:parabola}
\end{figure}

\subsection{Diminishing returns and epistasis.}\label{sect:epistasis}
In this subsection we summarize the heuristics which leads to the formula~\eqref{profix} for the fixation probability in our individual-based model, and compare it with the ansatz of Wiser et al. \cite{WRL}.

Our basic assumption is that every beneficial mutation adds a fixed amount $\varrho_N$ to the reproduction rate $r$ of the individual that undergoes the mutation. When all (or nearly all) individuals that are present  at day $i$ have reproduction rate $r$, then this day ends (approximately) at time $\sigma := \frac{\log \gamma}{r}$, because then $e^{r\sigma} =  \gamma$. Consequently, over this day the growth factor of a mutant population  whose reproduction rate is $r+\varrho_N$ is $e^{(r+\varrho_N)\sigma}$, and the ratio of these two growth factors is $e^{\varrho_N \sigma}\approx 1+\frac{\varrho_N \log \gamma}{r}$, revealing that the {\em selective advantage} of the mutant is  $s_N := \frac{\varrho_N \log \gamma}{r}$. In the branching process approximation for the onset of the mutant, $1+s_N$ is the offspring mean, while the quantity $c_\gamma  = 2(1-\frac{1}{\gamma})$ that appeared already in Sec.\ref{neutralmodel} converges for $N\to \infty$ to the offspring variance, see the discussion after Theorem \ref{thm:kingman}. In view of Lemma \ref{lem:GW}, this explains the form~\eqref{profix}  of the fixation probability.

A related observation appears in \cite{Chevin}: if two populations grow (as a pure birth process) with Malthusian parameters $r_{\rm w}$ and $r_{\rm m}$, and if one generation corresponds to a doubling of the population size, then a ``correct measure for the dynamics of selection per generation'' is $(r_{\rm m}-r_{\rm w})T$, where $T=(\log 2)/ r_{\rm w}$ is the {\em generation time} (see \cite{Chevin}, formula (3.2)). Our model reflects such a generation scheme, with $\log \gamma$ instead of $\log 2$, due to the design of the Lenski experiment.

It is interesting to note that our model leads to quite similar conclusions as the one proposed in  \cite{WRL}, although the basic hypotheses are somewhat different. Motivated by \cite{Gerrish} the authors of  \cite{WRL} {\em assume} that the $(n+1)$-st successful mutation increases the individual reproduction rate by a factor $1+\hat S_{n+1}$, where $S_{n+1}$ is distributed exponentially with some parameter $\alpha_{n}$, and the distribution of $\hat S_{n+1}$ is that of  $S_{n+1}$ conditioned to the event that the mutation goes to fixation (surviving also clonal interference). They make the following assumption in order to model \emph{diminishing returns}: The sequence $\alpha_{n}, n\in\N_0$, satisfies
 \begin{equation}\label{eq:epistasis}
  \alpha_{n+1}=\alpha_n(1+g\langle S_{n+1}\rangle ),
 \end{equation}
where $g$ is a positive constant and $\langle S_{n+1}\rangle$ is the expected value of $\hat S_{n+1}$. According to \cite{WRL},  the parameter $g$ serves to model the phenomenon of {\em epistasis}, which corresponds to a non-linearity in the fitness effects. 
Through \eqref{eq:epistasis}, it is a priori assumed that the expected value of the beneficial effect of a mutation decreases as the number of successful mutations increases. Arguing heuristically  by a branching process approximation,
the authors of \cite{WRL} obtain an approximation of the relative fitness by the function 
\begin{equation}\label{wiser}
\overline{w}=(ct+1)^{1/2g}.
\end{equation}
Here  $c$ depends on clonal interference and epistasis. In \cite{WRL}  the approximation is compared to real data, taking different pairs $(g,c)$ and proving that the power law approximation in equation \eqref{wiser} fits better to data than the hyperbolic curve proposed in \cite{Barrick}.

Our Theorem \ref{thm:adapt} is consistent with \eqref{wiser}, as we prove that, under the assumptions of our model,
\begin{equation}
\overline{w}=(c't+1)^{1/2}.
\end{equation}
Notably, the ``diminishing returns'' for the case $g=1$ emerge in our model under the assumption that every  beneficial mutation adds a constant amount $\varrho_N$  to the  intraday individual reproduction rate, which corresponds to the absence of epistasis in this part of the model. This shows that the observed power law behaviour of the relative fitness can to some extent be explained by the mere design of the experiment, based on a simple non-epistatic intraday model -- a fact which may also be seen as a strengthening of  the argument of Wiser et al \cite{WRL} that a power law is an appropriate approximation to the evolution of relative fitness. 

 In order to arrive at a power law \eqref{wiser} for more general $g$, we have to extend our model slightly.
Indeed, in Corollary \ref{Epistasis} we prove that a gain in the reproduction rate of $x^{-q}{\varrho_N},$ for some $q>-1$, if the present relative fitness is $x$, leads to a power law fitness curve with exponent $1/(2(1+q)),$ which compares to \eqref{wiser} by taking $q=g-1$. 

For a recent study that proposes a general framework for quantifying patterns of macroscopic epistasis from observed differences in adaptability, including a discussion of fitness and mutation trajectories in the  Lenski experiment, see \cite{GoodDesai}.~We refer also to the discussion in \cite{CouceTenaillon} of various epistatic models that would explain a declining adaptability in microbial evolution experiments,  and to the discussion in \cite{epistasisadaptation} concerning the evolutionary dynamics on epistatic versus non-epistatic fitness landscapes with finitely many genotypes.

\section{Models and main results}\label{sec:modelandresults}

\subsection{Mathematical model of daily population cycles}\label{sect:daily}
In this section, we  construct a mathematical model for the daily reproduction and growth cycle of a bacterial population in the Lenski experiment, and state some first results,
in particular on fixation probabilities of beneficial mutations. These are the foundations for our main results to be presented in Section \ref{sect:adaptation}.

\subsubsection{Neutral model}\label{sect:neutral}
We start by introducing the neutral model, where all individuals in the population reproduce at the same rate. The model consists of a continuous time intraday dynamics, and a discrete time interday dynamics, the latter is governed by a stopping- and a sampling rule.
We number the daily cycles, or ``days'' as we call them for simplicity, by $i\in\N_0$. Fix $N\in\N,$ and  $r>0.$ We assume that every daily cycle starts with exactly $N$ individuals that reproduce at rate $r,$ the \emph{basic reproduction rate}. More precisely, we decree that, 
independently for every day $i\in\N_0,$ the (neutral) intraday \emph{population size process} has the distribution of a \emph{Yule process}, denoted by $(Z^{(N)}_t)_{ t\geq 0},$ with reproduction parameter $r,$ started with $Z^{(N)}_0=N$ individuals. Consequently, for every $t> 0,$ the random variable $Z_t^{(N)}$ follows a negative binomial distribution with parameters $N$ and $e^{-rt}$ (see Corollary A.4 in Appendix A). In Appendix \ref{app:yule}, we collect the properties of Yule process that are relevant for this paper.

Fix now $\gamma>1,$ and define stopping times

\be \label{tau} \varsigma_N:=\inf\{t>0: Z_t^{(N)}\geq \gamma N\}\ee
and

\be  \sigma^{(N)}:=\inf\{t>0: \E[Z_t^{(N)}]\geq \gamma N\}.\ee

Note that $\varsigma_N$ is a random variable, while $\sigma^{(N)}$ is deterministic. In fact, since $\E[Z_t^{(N)}]= Ne^{rt},$ we see immediately that $\sigma^{(N)}$ does not depend on $N$ and equals
\be \sigma=\frac{\log \gamma}{r}. \ee

\begin{definition}[Neutral model] \label{def:neutral} 
Fix $N\in\N,\, r>0,\,\gamma>1.$ In the neutral model, independently for every $i\in\N_0,$ the population size at the end of day $i$ is given by a copy of the random variable $Z_{\sigma}^{(N)},$ where $(Z_t^{(N)})_{t\geq 0}$ is defined above.
\end{definition}

In other words, at every day the neutral population is started with $N$ individuals that reproduce by binary splitting at rate $r$ (which leads to the above Yule process), with  the population growth stopped at time $\sigma$ that depends on $\gamma$ and $r$.

\begin{remark}[Stopping rules]
The two stopping times $\varsigma_N$ and $\sigma$ give rise to two different stopping rules for the population: The \emph{stopping rule $1$} stops the population growth at time $\varsigma_N,$ that is the time when population size has reached exactly $\lceil \gamma N\rceil.$ On the other hand, \emph{stopping rule 2} uses $\sigma$ instead, which implies that the size of the stopped population, given by $Z_\sigma^{(N)}$, has a negative binomial distribution with parameters $N$ and $\frac{1}{\gamma}$.
While $\varsigma_N$ might be a more natural choice for the stopping time of the population growth, $\sigma$ is easier to deal with. In this paper we will work under stopping rule 2, but we  expect the essentials of our results to be true for $\varsigma_N$ as well. In fact, as we show in Lemma \ref{lem:conv_time}, $\varsigma_N$ converges to $\sigma$ in distribution.
\end{remark}

{\subsubsection{The genealogy}
Before turning our attention to the model with selection, we briefly discuss the neutral genealogy.}
If we label the individuals within this process, we can keep track of their ancestral relationship by specifying a sampling rule.
\begin{definition}[Sampling rule]\label{def:sampling}
The parent population of day $i+1$ is a uniform sample of size $N$ taken from the population at the end of day $i$.

\end{definition}
Let $\nu^i=(\nu_1^i,\cdots,\nu_N^i), \, i=0,1,2, \ldots$, be a sequence of vectors such that $\nu_j^i$ is the number of offspring in the population at the beginning of day $i$ of individual $j$ from the population at the beginning of day $i-1$. Since $(\nu^i)_{i\in\N_0} $ are independent and identically distributed, and for each $i$ the components of $\nu^i$ are exchangeable and sum to $N$, we are facing a Cannings model, where the ``days'' play the role of generations (see \cite{wakeley} for more background on Cannings models and coalescents).
We can now fix a generation $i$ and consider the genealogy of a sample of $n (\leq N)$ individuals. Here, for conceptual and notational convenience, we shift the ``present generation'' to the time origin and extend the Cannings dynamics (which is time-homogeneous) to all the preceding generations as well. 

\begin{definition}[Ancestral process]
Sample $n$ individuals at generation $0$ and denote them by $l_1,\cdots,l_n$. Let $\mathcal{P}_n$ be the set of partitions of $\{1,2,\cdots,n\}$ and $B^{(N,n)}=(B^{(N,n)}_g)_{g\in\mathbb{N}_0}$ be the process taking values in $\mathcal{P}_n$ such that any $j,k$ being in the same block in $B^{(N,n)}_g$ if and only if there is {a common} ancestor at generation $-g$ for individuals $l_j,l_k$. Then $B^{(N,n)}$ is the ancestral process of the chosen sample.
\end{definition}

It turns out that the genealogical process converges after a suitable time-scaling to the classical Kingman coalescent (see \cite{wakeley} for a definition and more details on the relevance of Kingman's coalescent in population genetics). The time-rescaling depends on the population size $N$ and is determined by a constant depending on $\gamma.$
\begin{theorem}[Convergence to Kingman's coalescent]\label{thm:kingman} For all $n \in \mathbb N$, the sequence of ancestral processes 
$\big(B^{(N,n)}_{\lfloor Nt/2(1-\frac{1}{\gamma}\big)\rfloor})_{t\geq 0}$ converges weakly on the space of c\`adl\`ag paths as $N\to\infty$ to Kingman's $n$-coalescent.
\end{theorem}

The proof of Theorem \ref{thm:kingman} 
is given in Appendix A. Here we give a brief heuristic explanation of the time change factor $2(1-1/\gamma)/N$. This factor is asymptotically equal to $c_{\gamma,N}$,  the {\em pair coalescence probability in one generation}, which in turn equals the probability that the second of two sampled individuals belongs to the same (one generation) offspring as the first one. Hence, in the limit $N \to \infty$, $c_{\gamma,N}$ is asymptotically equal to the ratio $(\mathbb E \hat G  -1)/(N\mathbb E G)$, where $G$ is the one-generation offspring number of a single individual, and $\hat G$ is a size-biased version of $G$. If $G$ has a geometric distribution with expectation $\gamma$ (which is the case in our setting, as can be seen from Lemma \ref{xgeo} in the Appendix), then $\mathbb E \hat G = \mathbb E G^2 /  \mathbb E G = 2\gamma -1$, and hence $c_{\gamma,N} \sim  2(1-\frac 1\gamma)/N$.
(In particular, for large~$\gamma$, $G/\gamma$ is asymptotically exponential, $\mathbb E\hat G \sim 2\mathbb EG$, and $c_{\gamma, N} \sim \frac 2 N$.)

\subsubsection{Including selective advantage}\label{selection}
We now drop the assumption that the relative fitness is constant over the whole population, and include some selective advantage. Fix $r>0, \gamma>1$ as before. For $N\in\N$ let $\varrho_N\geq 0.$ Throughout this paper, we will assume that the sequence $(\varrho_N)_{N\in\N}$ satisfies the condition
\be\label{sNpolynom} \exists b\in(0,1/2): \varrho_N\sim N^{-b} \mbox{ as } N\to\infty.\ee 

We extend our basic population model in the following way. Assume that at day $i$ a number~$k$ among the $N$ individuals of the initial population have  a selective advantage in the sense that they reproduce at rate $r+\varrho_N,$ and the remaining $N-k$ individuals reproduce at rate $r.$ We call the selectively advantageous individuals the \emph{mutants}, and the others the wild-type individuals. We assume that fitness is heritable, meaning that offspring (unless affected by a mutation) retain the fitness of their parent. The intraday population size process at day $i$ is then of the form

\be \label{defY}
{Y}_t:= {Y}_t^{(N,k)}= M_t^{(k)}+Z_t^{(N-k)}, \quad t\ge 0,
\ee
where $(Z_t^{(N-k)})_{t\geq 0}$ is a Yule process with reproduction rate $r$, started with $Z_0^{(N-k)}=N-k$ individuals, while 
$(M_t^{( k)})_{t\geq 0}$ is a Yule process with reproduction rate $r+\varrho_N$, started with $M_0^k=k$ individuals, and independent of $(Z_t^{(N-k)})_{t\geq 0}.$ Note that for fixed $r$ and $\varrho_N$ the distribution of $({Y}_t)_{t\geq 0}$ is uniquely determined by the initial number $M_0^{(k)}=k$ of mutants.

\medskip

We apply stopping rule 2 to this model, which translates into stopping population growth at a deterministic time depending on $k$ (and $N$), namely at
\be \label{sigmak} \sigma_{k}:= \sigma_{k}^{(N)}= \inf\{t\geq 0: \E[Y_t]\geq \gamma N\}.\ee
This is still a deterministic time, though somewhat harder to calculate than $\sigma,$ which equals $\sigma_0$ in this notation. 
Due to our construction, at the end of day $i$ the total population has size ${Y}_{\sigma_{k}}$,
among which there are $M_{\sigma_{k}}^{(k)}$ mutants, and $Z_{\sigma_{k}}^{(N-k)}$ wild-type individuals. 
 
 \medskip
 
One of the main tasks of this paper will be to calculate the \emph{number of mutants} at the beginning of day $i,$ for $i\in\N_0.$ Assuming that we know the population $Y_{\sigma_k}=M_{\sigma_{k}}^{(k)}+Z_{\sigma_{k}}^{(N-k)}$ at the end of day $i-1,$ we apply Definition \ref{def:sampling}, which means that given $M_{\sigma_k}^{(k)}=M,$ and $ Z_{\sigma_k}^{(N-k)} =Z,$ we sample uniformly $N$ out of the $M+Z$ individuals. Denote by $K_i$ the number of mutants contained in this sample. Fixing $K_0$ and repeating this independently for $i\in\N$ defines the interday process $(K_i)_{i\in\N_0}$ counting the number of mutants in the model with selection at the beginning of each day. Summarizing, this process can be described as follows:

{\begin{proposition}[Model with selection]\label{def:selection}
Fix $\gamma>1, r>0$ and $\varrho_N, N\in\N$ satisfying \eqref{sNpolynom}. Fix $K_0\in\{1,...,N\}.$ Assume $K_{i-1}$ has been constructed, and takes the value $k$. Let $M$ follow a negative binomial distribution with parameters $k$ and $e^{-(r+\varrho_N)\sigma_k},$ and let $Z$ follow a negative binomial distribution with parameters $N-k$ and $e^{-r\sigma_k}$ independent of $M$. Conditional on $M$ and $Z,$ the number $K_i$ is determined by sampling from the hypergeometric distribution with parameters $N, M$ and $M+Z.$
\end{proposition} 

\begin{proof}
This follows from the construction, noting that $(M_t)_{t\geq 0}$ and $(Z_t)_{t\geq 0}$ evolve independently until the deterministic time $\sigma_k,$ and recalling that sampling $N$ individuals without replacement out of $M$ of one type and $Z$ of another type is described by the hypergeometric distribution.
\end{proof}

\begin{remark}[More than two types]
The definition of the model with selection generalizes in an obvious way to situations where there are more than two different types of individuals in the population. If there are $\ell$ different types reproducing at $\ell$ different (fixed) rates, the population within one day grows like $\ell$ independent Yule processes with suitable initial values and reproduction rates, the stopping time is defined accordingly, and the sampling remains uniform over the whole population.
\end{remark}

Since the mutants reproduce faster, their proportion will increase (stochastically) during the day. Hence, sampling uniformly at random from the population at the end of day $i$ we expect to sample more than the initial number of mutants, meaning that the fitness of the 
population will increase over time.
 
\begin{proposition}[Selective advantage]\label{prop:selective_advantage} Under assumption \eqref{sNpolynom}, 
\begin{equation}\E[K_1|K_0=1]-1\sim \varrho_N\frac{\log \gamma}{r} \quad \mbox{as }\, N\to\infty.\end{equation}
\end{proposition}
Under the condition $\{K_0=1\}$ the $N-K_1$ wild-type individuals that are sampled at the end of day~$0$ are exchangeably distributed upon the $N-1$ wild-type ancestors that were present at the beginning of day $0$. Hence, the expected (sampled) offspring of each of these wild-type ancestors is $\sim 1$ as $N\to \infty$, and thus, in view of Proposition \ref{prop:selective_advantage},
we can say that the \emph{selective advantage} of a single mutant, resulting from the increase of its reproduction rate from $r$ to $r+\varrho_N$, is given by $\varrho_N\frac{\log \gamma}{r}.$

\medskip

The main result of this section concerns the fixation probability of a beneficial mutation affecting one individual at the beginning of day 0, and an estimate of the time that it takes for a successful mutation to go to fixation (or for an unsuccessful mutation to go extinct). Let 

\be \pi_N:=\P\big(\exists i\in\N : K_i =N\,|\, K_0 =1\big)\ee
denote the probability of fixation if the population size process is started with one mutant at day 0 and write
\be\label{def:tfix}\tau_{\rm fix}^N:= \inf\{i\ge 1: K_i=N\}\in [0,\infty]\ee
for the time of fixation, and 
\be\tau^N_{\rm ext}:=\inf\{i\geq 1: K_i=0\}\in[0,\infty]\ee
for the time until the mutation has been lost from the population, with the usual convention that $\inf\emptyset =\infty.$ Let 
\[\tau^N:=\tau_{\rm fix}^N\wedge\tau_{\rm ext}^N\]
be the first day at which either the whole population carries the mutation, or there are no more individuals in the population carrying the mutation. Let
\be \label{def:c_gamma}
C(\gamma):=\frac{\gamma\log \gamma}{\gamma-1}.
\ee

\begin{theorem}[Probability and speed of fixation]\label{thm:fixation}
Assume \eqref{sNpolynom}, {and assume that a mutation affects exactly one individual at day 0, and that no further mutations happen after the first one.} Then as $N\to\infty,$
\be \label{probfix} \pi_N\sim \varrho_N\frac{C(\gamma)}{r}.\ee
 Moreover, 
for any $\delta>0$ there exists $N_\delta\in \N$ such that for all $N\geq N_\delta$ 
\be\label{timefix}\P(\tau^N>\varrho_N^{-1-3\delta})\leq (7/8)^{\varrho_N^{-\delta}}.\ee
\end{theorem}

The proof, which will be given in Section \ref{proofsec}, relies on a comparison with a  supercritical (near-critical) Galton-Watson process in the ``early phase of the sweep''. While the basic idea is classical (dating back to work of Fisher from the 1920's), the scaling \eqref{sNpolynom} of the supercriticality and the specific nature of our Cannings dynamics required new arguments and a delicate analysis. For related results on near-critical Galton-Watson processes (which in some parts inspired our reasoning) see the recent work of Parsons \cite{Parsons}.

\subsection{Genetic and adaptive evolution}\label{sect:adaptation}
Our ultimate goal is to understand the deceleration of the increase in the relative fitness observed in \cite{WRL}, in particular as compared to the linearly increasing number of successful mutations (``adaptive versus genetic evolution''). In our model the relevant scales for the two processes turn out to be different, since the assumptions are such that many successful mutations are needed in order to have a change of approximately one unit in the relative fitness. 
 
This section is divided into two parts. First, we study the model on a short time scale, which is the relevant one for the arrivals of successful mutations. We prove that under some assumptions on the model parameters the number of successful mutations converges on a suitable time scale to a standard Poisson process. Afterwards, we introduce the process of relative fitness of the population, and we show that this process converges on a longer time scale to a deterministic function. 
 
\subsubsection{Genetic and adaptive evolution on a short scale}\label{sect:short_scale}
The assertion of Theorem \ref{thm:fixation} can be rephrased as follows:
\textit{In a background of wild-type individuals that reproduce at rate $r$, a beneficial mutation that leads to a reproduction rate $r+\varrho_N$ has a probability of fixation obeying 
\eqref{probfix}.} Besides recalling condition~\eqref{sNpolynom} on the selection, in the following assumption we require that the mutation rate is small enough to exclude ``effective clonal 
interference'' between beneficial mutations.

\begin{assumption}[Additive, moderately strong selection-weak mutation]
Beneficial mutations occur and act in such a way that the following hold:
\begin{itemize}
\item [i)] Beneficial mutations add $\varrho_N$ to the reproduction rate  of the individual that suffers the mutation.
\item [ii)] In each generation, with probability $\mu_N$ there occurs a beneficial mutation.
 The mutation affects only one (uniformly chosen) individual, and 
every offspring of this individual also carries the mutation.  
\item [iii)] There exists $0<b<1/2,$ and $a>3b,$ such that $\mu_N\sim N^{-a}$ and $\varrho_N\sim N^{-b}$ as $N\to\infty.$
\end{itemize}
\end{assumption}

We use the term \emph{moderately strong selection} in order to indicate that the strength of selection in our model is between what is generally called \emph{strong selection}, where $\varrho_N=O(1),$ and \emph{weak selection} where $\varrho_N=O(N^{-1})$ as $N\to\infty.$ Models with such types of selection were recently considered in the context of density dependent birth-death-mutation processes by Parsons \cite{Parsons, parsons2008absorption}. 
The term {\em weak mutation} is used to indicate that the mutation rate is small enough to guarantee the absence of clonal interference as $N\to \infty$, which we will prove in Proposition \ref{noclonalint}.

\begin{definition}[Interfering mutations, clonal interference] Consider a pair of successive mutations. Recall that $\tau^N$ denotes the first time after the first mutation at which the
individual reproduction rate is constant within the population. Denote by $m_N$ the time of the second mutation. We say that the two \emph{mutations interfere} if $m_N<\tau^N$, and that 
\emph{clonal interference} occurs if there exists a pair of interfering mutations. In particular, there is \emph{no clonal interference until day $i$} if there is no mutation starting until day $i$ that interferes with any other mutation.
\end{definition}

\begin{remark}
(i) As we will see in Proposition \ref{noclonalint} below, Assumption A iii) guarantees that the probability of clonal interference of any pair of successive mutations is of order at most $\mu_N \varrho_N^{-1}.$ In particular, this ensures that the probability of not observing any event of clonal interference on a time scale of order $\mu_N^{-1} \varrho_N^{-2}$ (which we will see to be relevant for our model) tends to 1 as $N\to\infty.$ 
\\ (ii) Our assumption A iii) is somewhat stronger than requiring $\mu_N\ll \varrho_N,$ which is a standard assumption in adaptive dynamics excluding clonal interference, see e.g. \cite{De13chemostatmodel}.
In view of Theorem \ref{thm:fixation} and of our detailed calculations in Section 3 we think that replacing $a>3b$ by $a>b$ in Assumption A iii) should still lead to the same results. However, there are substantial technical difficulties to consider in this case, since $a>b$ only excludes clonal interference of two successive mutations, but not on the longer time scales that are relevant for our results.
\\ (iii) While there is little doubt that there is clonal interference (of successive beneficial mutations) in the Lenski experiment \cite{MLB15}, it is noticeable that, as will be seen  in Theorem \ref{thm:adapt},  in order to qualitatively explain certain features of the experimental results on the relative fitness of the population, it is not mandatory to include clonal interference as a model assumption. Including clonal interference into the model will be one goal of our future research in this topic.
\end{remark}

\begin{proposition}[Probability of clonal interference]\label{noclonalint} In our model, for any $\delta>0$ there exists $N_\delta\in\N$ such that for all $N\geq N_\delta,$
\[\P(m_N<\tau^N)\leq \mu_N \varrho_{N}^{-1-\delta}.\]
In particular, under Assumption A iii), for any $T>0,$
\be \lim_{N\to\infty} \P(\mbox{no clonal interference until day } \lfloor \varrho_N^{-2}\mu_N^{-1}T\rfloor )= 1.\ee
\end{proposition}

A quantity of interest is the \emph{number of successful mutations} up to a given day. Let $H_i$ denote the number of eventually successful mutations that have started until day $i,$ with $H_0=0.$ Since mutations arrive independently at rate $\mu_N,$ and fixate with probability $\sim \frac{C(\gamma)\varrho_N}{r_0}$ (at least in the absence of clonal interference), we expect that successful mutations arrive at rate $\frac{C(\gamma)\mu_N\varrho_N}{r_0}.$ Indeed, Proposition \ref{noclonalint} allows us to make this rigorous.

\begin{theorem}[Process of successful mutations]\label{thm:mutations} Let $H_i, i\in\N,$ be the number of successful mutations initiated until day $i,$ with $H_0=0.$ Let $r_0>0$ be the reproduction rate of the population at day 0, and let $(M(t))_{t\geq 0}$ be a standard Poisson process. Under Assumption A, for any $T>0,$ the process 
$(H_{\lfloor (\varrho_N\mu_N)^{-1}t\rfloor})_{0\leq t\leq T}$ converges in distribution (with respect to the Skorokhod topology on the space of c\`adl\`ag paths) to $\big(M(\frac{C(\gamma)}{r_0}t)\big)_{0\leq t\leq T}.$
\end{theorem}

 \begin{figure}[ht]
\centering
\includegraphics[height=6.5cm]{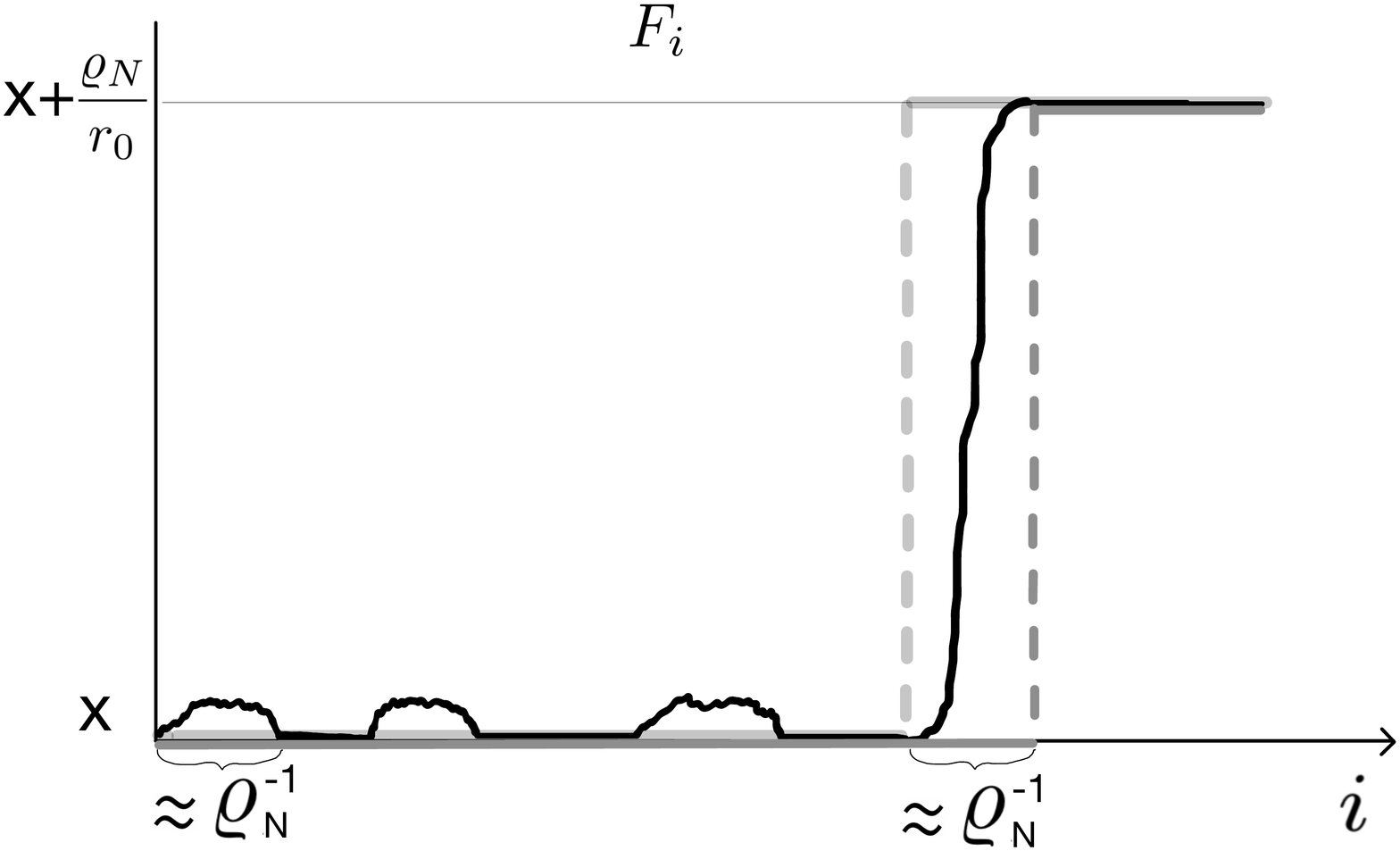}
\vspace{-0.5cm}
 \caption{\small The fitness process $F_i$ (solid black line), started at fitness $x,$ depicted until the time of fixation of the next successful mutation, in the absence of clonal interference. The light grey line represents the approximation $\Phi_i$ defined in (\ref{eq:rel-fit-mut}).}
 \label{fig:reprateprocesses}
\end{figure}

\subsubsection{Genetic and adaptive evolution on a long time scale}\label{sect:long}
Our next goal is to investigate the process describing the fitness of the evolved population relative to the ancestral population at day $0$. Let $R_{i,j},$ for $i\in\N_0$ and $1\leq j\leq N$, denote the reproduction rate of individual $j$ at the beginning of day $i.$ Assume that at day 0 every individual has reproduction rate $r_0,$ that is, $R_{0,j}=r_0$ for all $j=1,...,N.$ Recall from \eqref{eq2} the definition of the \emph{relative fitness} at day $i$ with respect to day 0. We can connect the relative fitness with the number of successful mutations in the following way. Let $\underline{R}_i:=\min_{1\leq j\leq N}R_{i,j}$ and $\overline{R}_i:=\max_{1\leq j\leq N}R_{i,j}$ denote the minimal and maximal reproduction rate at day $i,$ respectively. Then we have
\[\frac{\underline{R}_i}{r_0}\leq F_i\leq\frac{\overline{R_i}}{r_0}, \quad i\in\N_0.\]
Moreover, on the event that there is no clonal interference up to day $i,$ one has
\[r_0+\varrho_N(H_i-1)\leq \underline{R}_i\leq \overline{R}_i\leq r_0+\varrho_n H_i.\]

Let 
\begin{equation}\label{eq:rel-fit-mut}
\Phi_i:=1+\frac{\varrho_N}{r_0}H_i.
\end{equation}
Thus on the event that there is no clonal interference we have
\be \label{eq:FPhi}\Phi_i-\frac{\varrho_N}{r_0}\leq F_i\leq \Phi_i.\ee

From Theorem \ref{thm:mutations} we see that the relevant time scale for the successful mutations is given by $\mu_N^{-1}\varrho_N^{-1}.$ Since the selective advantage of a single mutation is of order $\varrho_N$ (cf. Proposition \ref{prop:selective_advantage}), in view of \eqref{eq:rel-fit-mut} it seems plausible that the time scale on which to expect a non-trivial limit of the fitness process is $\varrho_N^{-2}\mu_N^{-1}.$ This suggests that the relative fitness has to be considered on a  time scale different from that of the number of successful mutations.

Indeed our next theorem shows that the process $F:=(F_{\lfloor \mu_N^{-1}\varrho_N^{-2}t\rfloor})_{t\geq 0}$ has a non-trivial scaling limit, which turns out to be a deterministic parabola.
 
\begin{theorem}[Convergence of the relative fitness process]\label{thm:adapt} Assume $R_{0,j}=r_0$ for $j=1,...,N$, and let $(F_i)_{i\in\N_0}$ be the process of relative fitness. Then under Assumption A, the sequence of processes 
$(F_{\lfloor (\varrho_N^2\mu_N)^{-1}t\rfloor})_{t\geq 0}$ converges in distribution as $N\to\infty$ locally uniformly to the deterministic function
\[
f(t)=\sqrt{1+\frac{2C(\gamma)t}{r_0^2}}, \,\,t\geq 0\, .
\]

\end{theorem}

The proof of this theorem will be given in Section \ref{proof:adapt}. It relies on the fact that due to Proposition \ref{noclonalint} the relative fitness process $(F_i)_{i\in\N_0}$ can be approximated by the process $(\Phi_i)_{i\in\N_0}$ defined in \eqref{eq:rel-fit-mut}.

A similar result can be obtained if a beneficial mutation provides an advantage that depends on the current fitness level. For example, let us assume that a mutation that goes to fixation when the relative fitness is $x$, provides an increment  to the reprodution rate that is of the form
\begin{equation}\label{epistasis}
\varrho_N^{(x)}=\psi(x)\varrho_N
\end{equation}
for some continuous function $\psi: [1,\infty) \rightarrow \mathbb{R}^+$. As we will see in the next corollary, a special choice of $\psi$ as a power function leads to a fitness curver similar 
to \eqref{wiser}.

\begin{corollary}\label{Epistasis}
Under Assumption A and (\ref{epistasis}), let $F^\psi_{i}$ be the relative fitness of the population at day $i$ with respect to the ancestral population at time $0$. Then the process 
$(F^\psi_{\lfloor (\varrho_N^2\mu_N)^{-1}t\rfloor})_{t\geq 0}$ converges in distribution and locally uniformly as $N\to\infty$ to the deterministic function $h$ which is the solution of the differential equation
\[
\dot{h}(t)=\frac{\psi(h(t))^2C(\gamma)}{r_0^2h(t)},\,h(0)=1,\,t\geq 0.
\]
In particular, if $\psi(x)=x^{-q}$ for some $q>-1,$ then 
\begin{equation} \label{fitness-epistasis}h(t)=\Big(1+\frac{2(1+q)C(\gamma)}{r_0^2}t\Big)^{\frac{1}{2(1+q)}}, t\geq 0.\end{equation}
\end{corollary}

This is similar to the family of curves found in \cite{WRL}, see also the discussion in Section \ref{sect:epistasis}.

\section{Proof of the main results}\label{proofsec}
In this section, we provide the proofs of the results that we stated in Section \ref{sec:modelandresults}, in particular Theorem \ref{thm:fixation}, which is technically the most involved and requires several preparatory steps, which are carried out first. After these preparations, the proof of Theorem \ref{thm:fixation} will be carried out in Section \ref{subsect:proof}. The proofs of the other main results will be given in Sections \ref{subsect:proofnoclonal} through \ref{proof:adapt}.

It turns out that if the number of mutants reaches at least $\varepsilon N,$ for some $\varepsilon \in (0,1),$ then the mutation will fixate with probability tending to one as $N\to \infty$. Our strategy for proving Theorem \ref{thm:fixation} is thus  to divide the time between the occurrence of a mutation and its eventual fixation into three stages. For the case of a successful mutation this is depicted in Figure \ref{phases}.

\begin{figure}[ht]
    \centering
 \includegraphics[width=7cm]{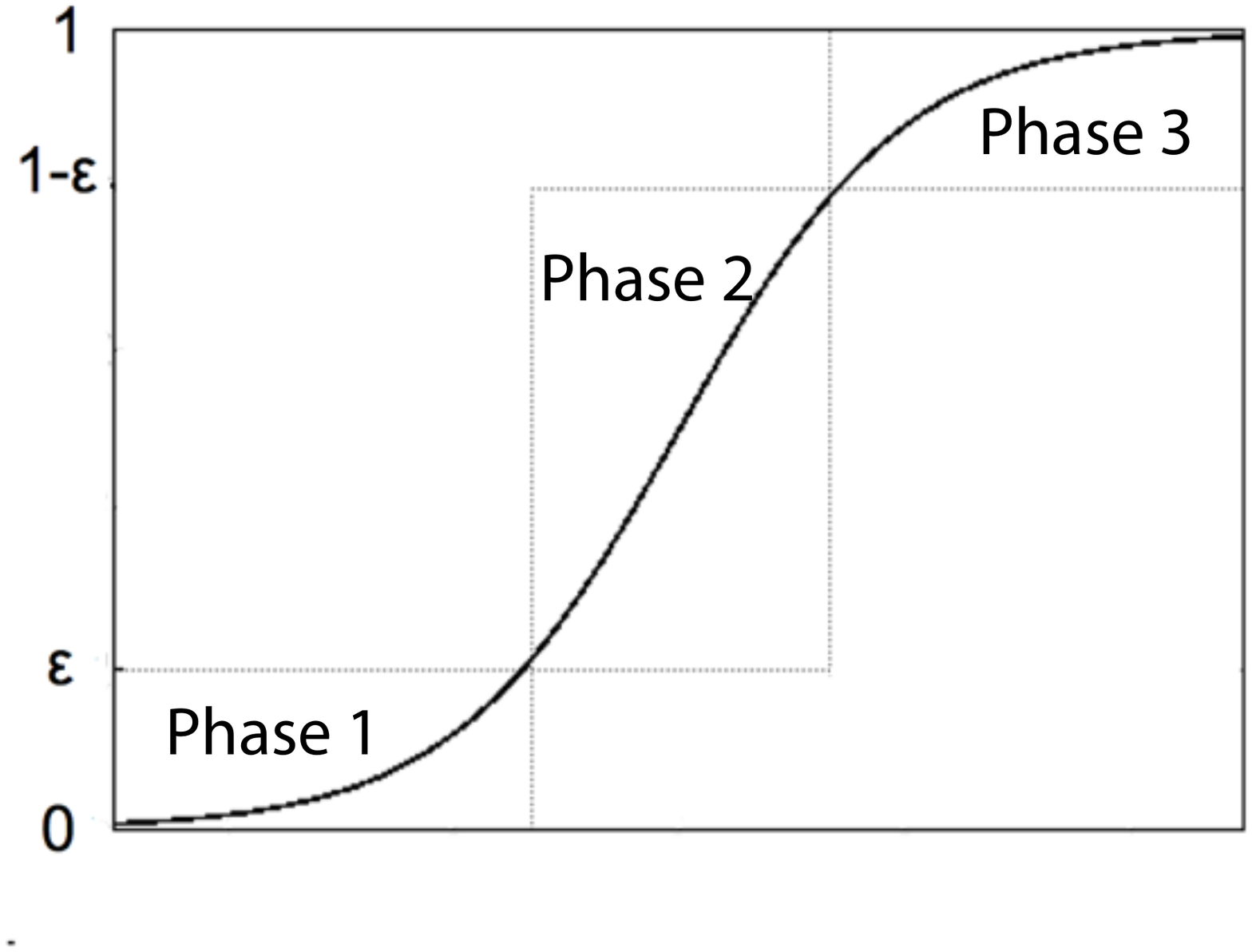}  
   \vspace{-2.5cm}
      \caption{\small A sketch of the  frequency of mutants during a selective sweep going to fixation. One distinguishes 3 parts: when the number of mutants is at most $\varepsilon N$ (phase 1),  when the number of wild type individuals is at most $\varepsilon N$ (phase~3) and  the intermediate stage (phase~2). This subdivision of a selective sweep into three phases is a now classical approach, see for example \cite{Kaplan, DurrettSchweinsberg04, Champagnat06}.
      }
      \label{phases}
\end{figure}
The first stage starts at the day of the mutation, and ends at the first day $i\in\N$ that the number $K_i$ of mutants has reached a level $\varepsilon N,$ for some $\varepsilon \in(0,1/2).$ The second stage starts upon reaching $\varepsilon N,$ and ends when the process $(K_i)_{i\in\N_0}$ reaches $(1-\varepsilon)N$. The last stage is between $(1-\varepsilon)N$ and $N.$ We will use different methods to analyze the behaviour of the process during these three stages. The first stage is the most difficult to deal with, and we use a coupling to suitable Galton-Watson processes to show that the probability that $(K_i)_{i\in\N_0}$ with $K_0=1$ ever reaches $\varepsilon N$ is approximated by \eqref{probfix}. The second stage can be treated by a simple ODE approximation, from which one sees that if $K_i\geq \varepsilon N$ at some time $i,$ then with probability tending to 1 (as $N\to \infty$) the process will eventually reach level $(1-\varepsilon)N.$ The third stage will be dealt with in a manner that has some 
similarities to the first stage, observing that starting from at least $(1-\varepsilon)N$ mutants, there is always a positive probability to reach fixation in the next step. Moreover, our methods of proof will also show that with high probability each of these stages will not last longer than $\varrho_N^{-1-\delta},$ for any $\delta>0.$
\\
To be more specific, fix $0<\varepsilon<1/2$. Assume $K_0=1.$ Let 
\[T^N_{1}:=\inf\{i: K_i\geq \varepsilon N\},\] 
and 
\[T^N_{2}:=\inf\{i: K_i\geq (1-\varepsilon) N\}.\]
Then we can write $\tau_{\rm fix}^N$ as the sum 
\begin{equation}\label{3-decom}\tau_{\rm fix}^N=T^N_1+(T^N_2-T^N_1)+(\tau_{\rm fix}^N-T^N_2).\end{equation}
The important intermediate steps of the proof, dealing with $T^N_1,$ $(T^N_2-T^N_1),$ and $(\tau_{\rm fix}^N-T^N_2)$, respectively, are given below in Sections \ref{subsect:firststage}, \ref{subsect:secondstage}, and \ref{subsect:thirsdstage}, after some preparatory steps in Sections \ref{subsect:GW-constr} through \ref{sect:asymGW}. The proof of Theorem \ref{thm:fixation} is completed in Section \ref{subsect:proof}.\\
\paragraph*{{\bf Assumption and notation.}}
Throughout all of Section \ref{proofsec} we fix $r>0, \gamma>1,$ and work under the assumption \eqref{sNpolynom}, fixing $b\in(0,1/2)$ accordingly. We a priori assume $\varepsilon\in (0,1/2),$ but note that in some places we will impose further conditions.
Unless stated otherwise, $\P_k, \E_k,$ and $\var_k, k\in\N,$ refer to the law, expectation and variance of $(K_i)_{i\in\N_0},$ started at $K_0=k,$ or any random variables defined on the same probability space. We use $c, c', \tilde{c},...$ to denote generic constants which are independent of $N,$ with possibly different values at different occurrences.

\subsection{A simplified sampling and construction of the auxiliary Galton-Watson processes}\label{subsect:GW-constr}
The construction of our model (as explained in Section \ref{selection}) was such that $K_{i+1}$ was obtained from $K_i$ by letting two independent Yule populations with initial sizes $K_i$ and $N-K_i$ and respective growth rates $r+\varrho_N$ and $r$ evolve until  time $\sigma_{K_i}$ (defined in \eqref{sigmak}) and then sampling uniformly $N$ individuals from the total of those two populations, which amounts to a mixed hypergeometric sampling of the number of individuals (Proposition \ref{def:selection}). 
In order to simplify the picture, we would like to use binomial rather than hypergeometric sampling, i.e. sampling individuals independently of each other with equal probability. 
In this way we will manage to construct two Galton-Watson processes $(\underline{K}_i)_{i\in\N_0}$ and $(\overline{K}_i)_{i\in\N_0}$ that will serve as upper and lower bounds for our true process $(K_i)_{i\in\N_0}$ in the first stage of the sweep. We prepare this construction by first giving an alternative description for the sampling of mutants.\\

Consider the population at the end of a given day (day 0, say). Assume $K_0=k,$ hence by construction at the end of day $0$ there are $M_{\sigma_k}$ mutant individuals for which we want to determine whether or not they will be sampled for the next day. (Recall the definition of $M_t$ from \eqref{defY}.) Label these mutant individuals with numbers $1,\dots, M_{\sigma_k}$.  Define
\[X_j:=1_{\{\mbox{\small individual $j$ is selected}\}},\quad  j=1,...,M_{\sigma_k}.\] 
Define a random variable 
\begin{equation}\label{defGamma}
\Gamma:=\frac{Y_{\sigma_k}}{N}.
\end{equation}
Thus $\Gamma$ is the ratio between the number of individuals at the end of day 0 and the number of individuals at the beginning of day 1, and by \eqref{sigmak}, $\E[\Gamma]=\gamma.$ Moreover, $\Gamma\geq 1,$ and $\P(\Gamma>1)$ is exponentially close to 1 as $N \to \infty$.
Conditional on $\Gamma,$ for every $j=1,...,M_{\sigma_k},$
\[\P(X_j=1)=\frac{1}{\Gamma},\] 
but due to our sampling mechanism, the $X_j, j=1,...,M_{\sigma_k}$, are not independent. Their joint law conditional on $\Gamma$ and $M_{\sigma_k}$ can be described as follows. Let $(U_j)_{j\in \N}$ be i.i.d uniform random variables on $[0,1]$. Let $\tilde{X}_1:=1_{\{U_1<1/\Gamma\}},$ and define recursively for $j\geq 2$
\be \label{def:Xtilde}\tilde{X}_j:=1_{\big\{U_j<\frac{N-\sum_{l=1}^{j-1}\tilde{X}_l}{\Gamma N-(j-1)}\big\}}.\ee
For later convenience we define $U_j$ and $\tilde{X}_j$ for $j\in\N,$ even though $X_j$ is defined only for $j=1,...,M_{\sigma_k}.$
\begin{lemma}\label{lem:coupling-success}
Conditional on $\Gamma,$  $(\tilde{X}_j)_{j=1,...,M_{\sigma_k}}$ is equal in distribution to $(X_j)_{j=1,...,M_{\sigma_k}}$.
\end{lemma}
\begin{proof}
Conditional on $\Gamma,$ we can represent the sampling procedure as follows: Individual 1 has probability $1/\Gamma$ of being selected. For individual 2, the probability of being sampled depends on whether or not individual 1 was selected, in fact
\begin{equation}
\P(X_2=1)=\frac{N-1}{\Gamma N-1}\P(X_1=1)+\frac{N}{\Gamma N-1}\P(X_1=0),
\end{equation}
or equivalently
\begin{equation}
\P(X_2=1|X_1)=\frac{N-X_1}{\Gamma N-1}.
\end{equation}
Proceeding thus recursively, we find that the probability that the $j$th individual is selected, conditional on knowing $X_1,...,X_{j-1}$, is
\begin{equation}
\P(X_j=1|X_1,...,X_{j-1})=\frac{N-\sum_{l=1}^{j-1}X_l}{\Gamma N-(j-1)}=\P(\tilde{X}_j=1|\tilde{X}_1,...,\tilde{X}_{j-1}).
\end{equation} 
This completes the proof.
\end{proof}

We can now construct the auxiliary Galton-Watson processes.
Fix $\alpha>0.$ We are going to specify a joint transition mechanism for $(K_i)_{i\in\N_0}$ and the auxiliary processes $(\underline{K}_i)_{i\in\N_0}$ and $(\overline{K}_i)_{i\in\N_0}$. To this purpose, let $\underline k, k, \overline k$ be natural numbers.
Grow independent Yule trees at rate $r+\varrho_N$ up to time $\sigma_0$, and number these trees by $\ell =1,2\ldots$. 
 Number all the individuals in this forest at time $\sigma_0$ by $j=1,2,\ldots$ and denote the $j$-th individual by $\mathcal I_j$.
Let $(U_j)_{j\in\N}$ be a sequence of independent uniformly on $[0,1]$ distributed random variables, independent of the Yule processes. For $j\in\N$ define 
\[\overline{X}_j:=1_{\{U_j<1/\gamma+N^{-\alpha}\}}\]
and
\[\underline{X}_j:=1_{\{U_j<1/\gamma-N^{-\alpha}\}}.\]
Also, define  $\Gamma$ as in \eqref{defGamma}, and $\tilde{X}_j$ by \eqref{def:Xtilde}.
We put
\be \label{def:underK}\begin{split} \underline{L}:=|\{j: \mathcal I_j \mbox{ belongs to the first } \underline k \mbox{ trees and is born before time } \sigma_{\lceil \varepsilon N \rceil}, \mbox{ and } \underline{X}_j=1\}|,\\ L:=|\{j: \mathcal I_j \mbox{ belongs to the first } k \mbox{ trees and is born before time } \sigma_k, \mbox{ and } \tilde{X}_j=1\}|\\ \overline{L}:=|\{j:  \mathcal I_j \mbox{ belongs to the first } \overline k \mbox{ trees and is born before time } \sigma_{0}, \mbox{ and } \overline{X}_j=1\}|.\end{split}\ee
Let
\be J:=\inf\Big\{j:\frac{N-\sum_{l=1}^{j-1}\tilde{X}_l}{\Gamma N-(j-1)}\in \mathbb{R}\setminus \big[\frac{1}{\gamma}-N^{-\alpha},\frac{1}{\gamma}+N^{-\alpha} \big]\Big\}.\ee
By construction it is clear that for every $j\leq J$ 
\begin{equation}\label{coup}
\sum_{l=1}^j\underline{X}_l\leq \sum_{l=1}^j \tilde{X}_l\leq \sum_{l=1}^j\overline{X}_l.
\end{equation}
Thus if $\underline{k}\leq k\leq \overline{k},$ on the event $\{J\geq M_{\sigma_k}\}$ we have
\[\underline{L}\leq L\leq \overline{L}.\]
\begin{definition}\label{KK} Let  $(\underline{K}_i, K_i, \overline{K}_i)_{i\in\N_0}$ be a Markov chain whose transition probability from $(\underline k, k,\overline k)$ is the joint distribution of $(\underline{L},L,\overline{L})$ given by \eqref{def:underK}.
\end{definition}
By construction, the coordinate processes $(\underline{K}_i)_{i\in\N_0}, (K_i)_{i\in\N_0}$ and $(\overline{K}_i)_{i\in\N_0}$ are also Markov chains. We note in particular that because of Lemma \ref{lem:coupling-success} the dynamics of $(K_i)_{i\in\N_0}$ is the same as that described in Proposition \ref{def:selection}. \\

We will show in the next section that if $\alpha\in (b,1/2),$ then $\P(J>M_{\sigma_k})$ is exponentially close to one for any $k\leq \varepsilon N$. From this we will deduce that with high probability, for $\underline{K}_0=K_0=\overline{K_0}=1,$ we have
\be \label{eq:monotone}\underline{K}_i\leq K_i\leq \overline{K}_i\quad \forall i\leq T_1^N. \ee}
Note that by definition
\be \label{eq:monotone-2}\underline{K}_i\leq \overline{K}_i,\quad \forall i\in\N_0\ee
always holds.
The following characterization of $(\underline{K}_i)_{i\in\N_0}$ and $(\overline{K}_i)_{i\in\N_0}$ is immediate from the construction:
\begin{proposition} \label{def:GW-processes}
Let $\alpha>0$ as before, and $(\underline{K}_i, K_i, \overline{K}_i)_{i\in\N_0}$ as in Definition \ref{KK}. Then $(\overline{K}_i)_{i\in\N_0}$ is a Galton-Watson process whose offspring distribution is mixed binomial with parameters $\overline{M}$ and $\frac{1}{\gamma}+N^{-\alpha},$ where $\overline{M}$ is geometric with parameter $e^{-(r+\varrho_N)\sigma_0}.$ Similarly, $(\underline{K}_i)_{i\in\N_0}$ is a Galton-Watson process whose offspring distribution is mixed binomial with parameters $\underline{M}$ and $\frac{1}{\gamma}-N^{-\alpha},$ where $\underline{M}$ is geometric with parameter $e^{-(r+\varrho_N)\sigma_{\lceil \varepsilon N \rceil}}.$
\end{proposition}

\subsection{A Galton-Watson approximation}\label{subsect:coupling} 
A crucial role in our analysis of stage 1 of the sweep will be played by equation \eqref{eq:monotone}, which we are now going to prove. Let $b$ be such that \eqref{sNpolynom} holds, and assume $K_0=k$ for some $k\leq \varepsilon N.$ 
We will show that if $\alpha>b,$ then with sufficiently large probability $J>N$, and $M_{\sigma_k}<N.$ The first part will require some work. To start with, we will work with a slight modification of $J$.
Let 
\be \tilde{J}:=\inf\Big\{j:\frac{N-\sum_{l=1}^{j-1}\tilde{X}_l}{\Gamma N-(j-1)}\in \mathbb{R}\setminus \big[\frac{1}{\Gamma}-\frac{1}{2} N^{-\alpha},\frac{1}{\Gamma}+\frac{1}{2} N^{-\alpha} \big]\Big\}.\ee

\begin{lemma}\label{prop1}
Let $\alpha\in (b,1/2).$ There exists a constant $\tilde{c}$ independent of $N$ such that for $N$ large enough,
\begin{equation}\label{thetatilde}
\P\Big(\tilde{J}> N\,\Big|\,\big|\gamma-\Gamma\big |\leq \frac{1}{2}N^{-\alpha}\Big) \geq 1-2e^{-\tilde{c}N^{1-2\alpha}}.
\end{equation}
\end{lemma}
\begin{proof}

Let $A_\Gamma:=\big\{\big|\gamma-\Gamma\big |\leq \frac{1}{2}N^{-\alpha}\big\}.$ By the construction and the definition of $\tilde{X}_j$, equation \eqref{thetatilde} is equivalent to
\begin{equation}
\P\Big(\frac{N-\sum_{l=1}^{j-1}\tilde{X}_l}{\Gamma N-(j-1)}\in \big[\frac{1}{\Gamma}-\frac{1}{2}N^{-\alpha},\frac{1}{\Gamma}+\frac{1}{2}N^{-\alpha} \big] \forall j\in \{1,...,N\}\,\Big|\,A_\Gamma\Big)\geq 1-2e^{-{\tilde c}N^{1-2\alpha}}.
\end{equation}
Now rearranging the terms one gets that for $0\leq j\leq N-1$
\be
\frac{1}{\Gamma}-\frac{1}{2}N^{-\alpha}\leq \frac{N-\sum_{l=1}^j\tilde{X}_l}{\Gamma N-j}\leq \frac{1}{\Gamma}+\frac{1}{2}N^{-\alpha} \ee
is equivalent to

\be-\frac{1}{2}N^{1/2-\alpha}\big(\Gamma-\frac{j}{N}\big)\leq \frac{1}{\sqrt{N}}\sum_{l=1}^j\big(\tilde{X}_l-\frac{1}{\Gamma}\big)\leq \frac{1}{2}N^{1/2-\alpha}\big(\Gamma-\frac{j}{N}\big).
\ee

So our aim will be to show that with sufficiently large probability on the event $A_\Gamma$
\[\sup_{j\in \{0,1,2,...,  N-1\}} \{\frac{1}{\sqrt{N}}\sum_{l=1}^j(\tilde{X}_l-\frac{1}{\Gamma})\}\leq \frac{1}{2}N^{1/2-\alpha}(\Gamma-1)\] and 
\[\inf_{j\in \{0,1,2,...,  N-1 \}}\{\frac{1}{\sqrt{N}}\sum_{l=1}^j(\tilde{X}_l-\frac{1}{\Gamma})\}\geq -\frac{1}{2}N^{1/2-\alpha}(\Gamma-1).\]
Due to our assumptions, we can consider $(\overline{X}_j)_{j=0,...,N-1}$ resp. $(\underline{X}_j)_{j=0,...,N-1}$ instead of $(\tilde{X}_j)_{j=0,...,N-1}.$ Indeed, since $\gamma,\Gamma\geq 1$ we have on the event $A_\Gamma$
\be\label{eqgammaGamma}\big|\frac{1}{\gamma}-\frac{1}{\Gamma}\big|\leq \big|\gamma-\Gamma\big|\leq \frac{1}{2}N^{-\alpha}.\ee 
Then 
\[\big[\frac{1}{\Gamma}-\frac{1}{2}N^{-\alpha},\frac{1}{\Gamma}+\frac{1}{2}N^{-\alpha}\big]\subseteq \big [\frac{1}{\gamma}-N^{-\alpha},\frac{1}{\gamma}+N^{-\alpha}\big],\]
which implies that on the event $A_\Gamma,$ \eqref{coup} is valid for every $i\leq \tilde{J}$. We recall the independence between $A_{\Gamma}, \bar{X},\underbar{X}$. Thus we are done if we show
\be \label{eqsup}\P\Big(\sup_{j\in \{0,1,2,...,  N-1\}} \{\frac{1}{\sqrt{N}}\sum_{l=1}^j(\overline{X}_l-\frac{1}{\Gamma})\}\leq \frac{1}{2}N^{1/2-\alpha}(\Gamma-1)\,\Big|\, A_\Gamma\Big)\geq 1-e^{-\tilde{c}N^{1-2\alpha}}\ee and 
\be \label{eqinf}\P\Big(\inf_{j\in \{0,1,2,...,  N-1 \}}\{\frac{1}{\sqrt{N}}\sum_{l=1}^j(\underline{X}_l-\frac{1}{\Gamma})\}\geq \frac{1}{2}N^{1/2-\alpha}(-\Gamma+1)\,\Big|\, A_\Gamma\Big)\geq 1-e^{-\tilde{c}N^{1-2\alpha}}.\ee
This is an application of large deviations for maxima of sums of independent random variables, see for example \cite{Alesk}. Observing that $\E[\overline{X}_j]=\frac{1}{\gamma}+N^{-\alpha}$ and $\var(\overline{X}_j)=\gamma^{-1}(1-\gamma^{-1}+O(N^{-\alpha})),$ we obtain by a direct application of Theorem 1 of \cite{Alesk} that for any $A>0$ there exists $\tilde{c}_1=\tilde{c}_1(A,\gamma)\in(0,\infty)$ such that
\be \P\Big(\sup_{j\in \{0,1,2,...,  N-1\}} \{\frac{1}{\sqrt{N}}\sum_{l=1}^j(\overline{X}_l-\frac{1}{\gamma}-N^{-\alpha})\}> A N^{1/2-\alpha}\Big)\leq e^{-\tilde{c}_1N^{1-2\alpha}}.\ee 
Then \eqref{eqsup} follows with \eqref{eqgammaGamma}. Similarly we obtain \eqref{eqinf}.
\end{proof}

\begin{corollary}\label{GWbounds}
Let $\alpha\in (b,1/2)$. There exists a constant $c$ independent of $N$ such that for $N$ large enough
\begin{equation}
\P\big(J>N \big) \geq 1-e^{-cN^{1-2\alpha}}.
\end{equation}
\end{corollary}
\begin{proof}
Recall from the proof of the previous lemma that if $|\frac{1}{\gamma}-\frac{1}{\Gamma}|\leq \frac{1}{2}N^{-\alpha}$ then $$[\frac{1}{\Gamma}-\frac{1}{2}N^{-\alpha},\frac{1}{\Gamma}+\frac{1}{2}N^{-\alpha}]\subseteq [\frac{1}{\gamma}-N^{-\alpha},\frac{1}{\gamma}+N^{-\alpha}],$$ which implies that in this case $\tilde{J}<J$. We already observed that $|\frac{1}{\gamma}-\frac{1}{\Gamma}|\leq |\gamma-\Gamma|$, so it remains to show that $ |\gamma-\Gamma|<\frac{1}{2}N^{-\alpha}$ with large probability. Indeed, for $l=1,2,\dots,N$ and $N$ large enough
\begin{eqnarray*}
\P\big( |\Gamma-\gamma|\leq \frac{1}{2}N^{-\alpha}  \big)
&=& \P\big(|\frac{Y_{\sigma_l}-N\gamma}{\sqrt{N}}|\leq \frac{1}{2}N^{1/2-\alpha} \big)\\
&\geq &1- e^{-c^\prime N^{1-2\alpha}}
\end{eqnarray*}
for some constant in $c^\prime$ independent of $N$, where the last inequality follows from a generalisation of Cram\'er's theorem, see Theorem 2 of \cite{PR2008} (note that $\sigma_l$ is a sum of independent but not identically distributed random variables).  Let $c$ be a constant independent of $N$ such that $c>\max(c^\prime,\tilde{c}),$ where $\tilde{c}$ is the constant from Lemma \ref{prop1}. For $N$ large enough
\begin{eqnarray*}
\P\Big(J>N\Big) 
&\geq & \P\Big(\tilde{J}>N,|\frac{1}{\gamma}-\frac{1}{\Gamma}|\leq \frac{1}{2}N^{-\alpha}\Big)\\
&\geq & \P\Big(\tilde{J}>N \Big)-\P\Big(|\frac{1}{\gamma}-\frac{1}{\Gamma}|> \frac{1}{2}N^{-\alpha}\Big)\\
&\geq & 1- 2e^{-\tilde{c} N^{1-2\alpha}}- e^{-c^\prime N^{1-2\alpha}}\geq 1- e^{-c N^{1-2\alpha}}.
\end{eqnarray*}
\end{proof}
\begin{lemma}\label{lem:one-step}
Let $\alpha\in (b,1/2),$ and $0<\varepsilon <1/\gamma.$ Assume $\underline{K}_0\leq K_0\leq \overline{K}_0$ and $K_0=k,$ for some $k\leq \varepsilon N.$ There exists $c>0$ independent of $N$ such that for all $N$ large enough,
\[\P(M_{\sigma_k}<N)\geq  1-e^{-cN}.\]
\end{lemma}
\begin{proof}
Let $G_j$ be the number of offspring of the mutant number $j\leq k\leq \varepsilon N$ at the end of the day, namely at time $\sigma_k$. By construction they are i.i.d. with finite second moment. Let $(G'_j)_{j\in \mathbb{N}}$ be i.i.d random variables equal in distribution to $G_1.$ Note that $\E[G_1]\leq e^{(r+\varrho_N)\sigma_0}=\gamma(1+o(1)).$ Since $\varepsilon <1/\gamma$ we can choose $N$ large enough such that $\E[G_1]\leq 1/\varepsilon.$ Then 
\begin{eqnarray}\label{lessN}
\P\big( M_{\sigma_k}<N\big)=\P\big(\sum_{j=1}^kG_j<N\big)&\geq&\P\big(\sum_{j=1}^{\varepsilon N}G^\prime_j<N\big)\nonumber\geq 1-e^{-cN}
\end{eqnarray}
for a suitable $c>0$. The last inequality follows from Cramer's Theorem, since $\varepsilon \E[G_1]< 1.$
\end{proof}

Recall that $T^N_{1}=\inf\{i\geq 1: K_i\geq \varepsilon N\}.$ 

\begin{proposition}\label{GW1} 
Let $\alpha \in (b,1/2)$ and $0<\varepsilon<1/\gamma$. Assume $\underline{K}_0\leq K_0\leq \overline{K}_0$ and $K_0=k\leq \varepsilon N.$ Then there exists $c$ independent of $N$ such that for $N$ large enough
\begin{eqnarray}\label{eq:domination}
\P\big(\overline{K}_{\min (i,T^N_1)}\geq K_{\min (i,T^N_1)} \geq \underline{K}_{\min (i,T^N_1)},\forall i\leq g \Big)\geq (1-2e^{-c N^{1-2\alpha}})^g\quad \mbox{ for all } g\in\N_0.
\end{eqnarray}
\end{proposition}      

\begin{proof}
Corollary \ref{GWbounds} implies that $\P(\underline{K}_1\leq K_1\leq \overline{K}_1\; |\; M_{\sigma_k}<N)\geq1-e^{-c N^{1-2\alpha}}.$ Thus by Lemma \ref{lem:one-step} we have 
\be\label{eq:one-step} \P(\underline{K}_1\leq K_1\leq\overline{K}_1)\geq 1-2e^{-cN^{1-2\alpha}},\ee
which implies
\begin{eqnarray}\label{eq:induction}\notag
\P\big(\overline{K}_{\min (g,T^N_1)}\geq K_{\min (g,T^N_1)} \geq \underline{K}_{\min (g,T^N_1)}\,|\,\overline{K}_{\min (i,T^N_1)}\geq K_{\min (i,T^N_1)} \geq \underline{K}_{\min (i,T^N_1)},\forall i\leq g-1 \Big)\\ \geq 1-2e^{-c N^{1-2\alpha}}.
\end{eqnarray}
From \eqref{eq:induction} the result follows easily by induction:
Assume that \eqref{eq:domination} is true for $g-1$. Then
 \begin{align*}
\P &\big(\overline{K}_{\min (i,T^N_1)}\geq K_{\min (i,T^N_1)} \geq \underline{K}_{\min (i,T^N_1)},\forall i\leq g \Big)\\
&=
\P\big(\overline{K}_{\min (g,T^N_1)}\geq K_{\min (g,T^N_1)} \geq \underline{K}_{\min (g,T^N_1)}|\overline{K}_{\min (i,T^N_1)}\geq K_{\min (i,T^N_1)} \geq \underline{K}_{\min (i,T^N_1)},\forall i\leq g-1 \Big)\nonumber\\
&\hspace{0.3cm}\times \P\big(\overline{K}_{\min (i,T^N_1)}\geq K_{\min (i,T^N_1)} \geq \underline{K}_{\min (i,T^N_1)},\forall i\leq g-1 \Big)\\
&\geq  (1-2e^{-c N^{1-2\alpha}}) (1-2e^{-c N^{1-2\alpha}})^{g-1}.
\end{align*}
\end{proof}

\subsection{Asymptotics of the stopping rule}\label{subsect:stopping}

In order to put the Galton-Watson bounds to use, we need some control on $\sigma_k.$

\begin{lemma}\label{lem:sigma} Under the assumptions of this section, for any $k=1,2,\ldots,N,$
\be  \sigma_{k}=\frac{\log \gamma}{r+k\varrho_N/N}+\frac{k}{N}O(\varrho_N^2)+\frac{k^2}{N^2}O(\varrho_N^2).\ee
where $|O(\varrho_N^2)|/\varrho_N^2$ is bounded uniformly in $N$ and $k$. 
\end{lemma}

\begin{proof}
Note that $\frac{\log\gamma}{r+\varrho_N}=\sigma_{N}\leq \sigma_{k}\leq \sigma_0=\frac{\log\gamma}{r}$ for all $k=0,...,N.$ Hence $\lim_{N\to\infty}\sigma_{k}=\frac{\log\gamma}{r}$ for 
all $k.$ We assume that $N$ is large enough such that $\frac{\log\gamma}{2r}\leq \sigma_k\leq \frac{\log\gamma}{r}.$\\
By \eqref{defY} and \eqref{sigmak} we have

\be \begin{split}\gamma N=\E[M_{\sigma_k}^{(k)}]+\E[Z_{\sigma_k}^{(N-k)}]=ke^{(r+\varrho_N)\sigma_{k}}+(N-k)e^{r\sigma_{k}}.
\end{split}\ee
Hence $\sigma_k$ satisfies the equation
\be \gamma N=e^{r\sigma_{k}}\big(k e^{\varrho_N\sigma_{k}}+N-k\big).\ee
Dividing by $N,$ taking logarithms on both sides, and using Taylor expansion first on the exponential and then on the logarithm leads to

\be \begin{split}
\log \gamma =&r\sigma_{k}+\log\big(1+\frac{k}{N}\varrho_N\sigma_{k}+\frac{k}{N}O(\varrho_N^2)\big)\\
=& r\sigma_{k}+\frac{k}{N}\varrho_N\sigma_{k}+\frac{k}{N}O(\varrho_N^2)+\frac{k^2}{N^2}O(\varrho_N^2).\end{split}\ee
Here we use the fact that $\frac{\log\gamma}{2r}\leq \sigma_{k}\leq \frac{\log\gamma}{r}$ for {all $k$ if $N$ is sufficiently large}. Rewriting, we get the desired expression of $\sigma_k.$
\end{proof}

We will use this mostly in the following form, which is an immediate application of Lemma \ref{lem:sigma}.
\begin{corollary}\label{cor:sigma}
For any $k=1,2,\ldots,N,$ as $N\to\infty$
$$e^{(r+\varrho_N)\sigma_{k}}=\gamma\big(1+(1-\frac{k}{N})\frac{\varrho_N}{r}\log \gamma+O(\varrho_N^2)\big)$$
where $|O(\varrho_N^2)|/\varrho_N^2$ is bounded uniformly in $N$ and $k$. 
\end{corollary}

\subsection{Asymptotics of the approximating Galton-Watson processes and Proof of Prop.~\ref{prop:selective_advantage}}\label{sect:asymGW}
We can now calculate the asymptotic expectation and variance of our auxiliary Galton-Watson processes.  
\begin{lemma}\label{lem:expvarGW}
Let $\alpha\in(b,1/2).$ Let $(\underline{K}_i)_{i\in\N_0}$ and $(\overline{K}_i)_{i\in\N_0}$ be as defined in Section \ref{subsect:GW-constr} with $\underline{K}_0=K_0=\overline{K}_0=1.$ We have
\begin{equation}
\E_1[\overline{K}_1]=1+\frac{\log \gamma}{r}\varrho_N+o(\varrho_N)\, \qquad \E_1[\underline{K}_1]=1+\frac{\log \gamma}{r}(1-\varepsilon)\varrho_N+o(\varrho_N),
\end{equation} 
and
\begin{equation}
\var_1[\overline{K}_1]=\frac{2(\gamma-1)}{\gamma}(1+O(\varrho_N))\, \,\,\,\,\,\,\,\,\var_1[\underline{K}_1]=\frac{2(\gamma-1)}{\gamma}(1+O(\varrho_N)).
\end{equation}
\end{lemma}
\begin{proof}
Recall $\underline{M}, \overline{M}$ from Proposition \ref{def:GW-processes}. By construction, and from Corollary \ref{cor:sigma}
\begin{eqnarray*}\label{eq3}
\E_1[\underline{K}_1]&=&
(1/\gamma-N^{-\alpha})\E[\underline{M}]\\
&=&(1/\gamma-N^{-\alpha})e^{(r+\varrho_N)\sigma_{\lceil \varepsilon N \rceil}}\\
&=&1+\frac{\log \gamma}{r}(1-\varepsilon) \varrho_N-\gamma N^{-\alpha}+o(\varrho_N)\\
&=&1+\frac{\log \gamma}{r}(1-\varepsilon) \varrho_N+o(\varrho_N)
\end{eqnarray*} 
where the last equality follows from the fact that our assumptions imply that $ N^{-\alpha}=o(\varrho_N)$. In the same way we obtain
\[
\E_1[\overline{K}_1]=1+\frac{\log \gamma}{r} \varrho_N+o(\varrho_N).
\]
It remains to calculate the variance
\begin{eqnarray*}
\var_1[\underline{K}_1]&=&\E_1[\var_1[\underline K_1|\underline{M}]]+\var_1[\E_1[\underline K_1|\underline{M}]]\\
&=&\E_1\big([\underline{M}\big(\frac{1}{\gamma}-N^{-\alpha}\big)\big(1-\frac{1}{\gamma}+N^{-\alpha}\big)]+\var_1\big[\underline{M}\big(\frac{1}{\gamma}-N^{-\alpha}\big)\big]\\
&=&\big(\frac{1}{\gamma}-N^{-\alpha}\big)\big(1-\frac{1}{\gamma}+N^{-\alpha}\big)e^{(r+\varrho_N)\sigma_{\lceil \varepsilon N \rceil}}+\big(\frac{1}{\gamma}-N^{-\alpha}\big)^2\big(e^{2(r+\varrho_N)\sigma_{\lceil \varepsilon N \rceil}}-e^{(r+\varrho_N)\sigma_{\lceil \varepsilon N \rceil}}\big).
\end{eqnarray*}
Plugging in Corollary \ref{cor:sigma}, simplifying and taking into account that $N^{-\alpha}=o(\varrho_N)$ for $\alpha>b$ leads to
\[
\var_1[\underline{K}_1]=\frac{2(\gamma-1)}{\gamma}\big(1+(1-\varepsilon)\varrho_N\frac{\log \gamma}{r}+o(\varrho_N)\big)
=\frac{2(\gamma-1)}{\gamma}+O(\varrho_N).
\]
The same steps lead to $\var_1[\overline{K}_1]=2(\gamma-1)/\gamma+O(\varrho_N).$
\end{proof}

\begin{remark}\label{survive}
(i) This result together with Lemma \ref{lem:one-step} proves Proposition \ref{prop:selective_advantage}.
(ii) Applying Lemma~\ref{lem:GW} from the Appendix shows
\[\P((\overline{K}_i) \mbox{ survives})\sim \frac{C(\gamma)}{r}\varrho_N\]
and 
\[\P((\underline{K}_i) \mbox{ survives})\sim \frac{(1-\varepsilon)C(\gamma)}{r}\varrho_N.\]
\end{remark}

\begin{corollary}\label{cor:probGW}
Under the assumptions of Lemma \ref{lem:expvarGW}, for $k\leq \varepsilon N,$ as $N\to\infty,$
\be\P_k((\overline{K}_i)\mbox{ survives }|\,(\underline{K}_i)\mbox{ survives })=\P_k((\underline{K}_i)\mbox{ dies out}\,|\,(\overline{K}_i)\mbox{ dies out})=1.\ee
Further, 
\be \label{probtwoGW} \P_k((\underline{K}_i)\mbox{ dies out }|\,(\overline{K}_i)\mbox{ survives })\leq \varepsilon (1+o(1)),\ee
and 
\be\label{probtwoGW2}\P_k((\overline{K}_i)\mbox{ survives }|\,(\underline{K}_i)\mbox{ dies out })\leq \varepsilon (1+o(1)).\ee
\end{corollary}

\begin{proof} The first equation follows immediately from \eqref{eq:monotone-2}. We prove \eqref{probtwoGW}, \eqref{probtwoGW2} follows similarly.
Let $c(\gamma, r):=\frac{\gamma\log \gamma}{(\gamma-1)r}.$ 
Note that 
\begin{align}\label{kkk}
\P_k((\underline K_i)\mbox{ dies out} \,|\, (\overline{K}_i)\mbox{ survives })=&\frac{\P_k((\underline{K}_i)\mbox{ dies out})-\P_k((\overline{K}_i)\mbox{ dies out})}{\P_k((\overline{K}_i)\mbox{ survives})}\nonumber\\
\sim &\frac{(1-c(\gamma, r)(1-\varepsilon)\varrho_N)^k-(1-c(\gamma, r)\varrho_N)^k}{1-(1-c(\gamma, r)\varrho_N)^k}.
\end{align}
Let $g(k)$ be the r.h.s of (\ref{kkk}). We will show below that $g$ is decreasing in $k$ if $N$ is large, from which the statement follows, observing
\[g(k)\leq g(1)\leq \varepsilon(1+o(1)).\]
To prove the monotonicity of $g(k),$ let $a=c(\gamma,r)\varrho_N$. Let $N$ large enough such that $0<a<1$.  Assume that $k\geq 1$ and $k\in\R^+$. Then we
can differentiate $\log(1-g(k))$ in $k$ which yields

\begin{equation}\label{sign}\frac{d\log(1-g(k))}{dk}=\frac{(1-a)^k\log(1-a)}{1-(1-a)^k}-\frac{(1-a+a\varepsilon)^k\log(1-a+a\varepsilon)}{1-(1-a+a\varepsilon)^k}.\end{equation}
 The function $\frac{x^k log(x)}{1-x^k}$ is a decreasing function in $x,$ for $0<x< 1,$ as can be seen by differentiation.
Apply this to the r.h.s of (\ref{sign}), we obtain $\frac{d\log(1-g(k))}{dk}\geq 0$ for all $k\geq 1$. This implies $\frac{dg(k)}{dk}\leq 0$. So $g(k)$ is decreasing in $k.$
\end{proof}
\subsection{First stage of the sweep}\label{subsect:firststage}
With these preparations we can now address the first stage of the sweep, cf. Figure \ref{phases}. We are going to calculate the probability that the number of mutants reaches $\varepsilon N$ for some $\varepsilon>0,$ and determine the time it takes to reach $\varepsilon N$. We achieve this by using the supercritical Galton-Watson processes provided by Lemma \ref{GW1}. Recall $T_1^N=\inf\{i\geq 0: K_i\geq \varepsilon N\}.$

\begin{lemma}\label{lem:part3}
Let $0<\varepsilon<1/\gamma.$ Then we have as $N\to\infty$
\be\label{eq:bounds_firststage}\frac{\varrho_N\log \gamma}{r}\frac{\gamma}{\gamma-1}(1-\varepsilon)(1+o(1))\leq \P_1(\exists i: K_i \geq \varepsilon N)\leq \frac{\varrho_N\log \gamma}{r}\frac{\gamma}{\gamma-1}(1+o(1)),\ee
and for any $\delta >0$ 
\[\limsup_{N\to\infty}\P_1(T_1^N>\varrho_N^{-1-\delta}\,|\, T_1^N<\infty)\leq \frac{\varepsilon}{1-\varepsilon} .\]
\end{lemma}
\begin{proof}
Let $\alpha\in (b,1/2)$ and let $(\underline{K}_i)_{i\in\N_0}$ and $(\overline{K}_i)_{i\in\N_0}$ be defined as in Section \ref{subsect:GW-constr}, with $\underline{K}_0=K_0=\overline{K}_0=1.$ We write $(K_i)$ \emph{reaches} $\varepsilon N$ for the event that there exists $i>0$ such that $K_i\geq \varepsilon N,$ and analogously for $(\underline{K}_i), (\overline{K}_i)$. By Remark \ref{survive}, Lemma \ref{lem:GW}, and Lemma \ref{Prop:criticalGW}, 
\be\label{eq:upreach}\P_1((\overline{K}_i)\mbox{ reaches }\varepsilon N)\sim\P_1((\overline{K}_i)\mbox{ survives})\sim \frac{\varrho_N\log \gamma}{r}\frac{\gamma}{\gamma-1}\ee and 
\be\label{eq:lowreach}\P_1((\underline{K}_i)\mbox{ reaches }\varepsilon N)\sim\P_1((\underline{K}_i)\mbox{ survives})\sim \frac{\varrho_N\log \gamma}{r}\frac{\gamma}{\gamma-1}(1-\varepsilon).\ee

Let 
\[A:=A(\gamma,\alpha, \varepsilon, \delta, N):=\{\underline{K}_i\leq K_i\leq \overline{K}_i\, \,  \forall i\leq \min(T_1^N, \varrho_N^{-1-\delta})\}.\]
Setting $g:= \varrho_N^{-1-\delta}$ in Proposition \ref{GW1} and applying the Bernoulli inequality we have 
 \be \label{eq:expfast}\P_1(A^c)\leq 1-(1-2 e^{-cN^{1-2\alpha}})^{\varrho_N^{-1-\delta}}\leq \varrho_N^{-1-\delta}2e^{-cN^{1-2\alpha}},\ee
implying $\P_1(A)\to 1$ exponentially fast as $N\to\infty.$ Let $\underline{T}_1^N:=\inf\{i>0:\underline{K}_i\geq \varepsilon N\}.$ Then
\begin{eqnarray}\label{eq:asymeq}
\P_1((K_i) \mbox{ reaches } \varepsilon N)&\geq &\P_1((K_i) \mbox{ reaches } \varepsilon N, (\underline{K_i}) \mbox{ reaches }\varepsilon N, A, \underline{T}_1^N\leq \varrho_N^{-1-\delta})\nonumber\\
&=&\P_1((\underline{K_i}) \mbox{ reaches }\varepsilon N, A, \underline{T}_1^N\leq \varrho_N^{-1-\delta})\nonumber\\
&\geq  & \P_1((\underline{K_i}) \mbox{ reaches }\varepsilon N, \underline{T}_1^N\leq \varrho_N^{-1-\delta})-\P(A^c)\nonumber\\
&\sim& \P_1((\underline{K_i}) \mbox{ reaches }\varepsilon N)
 \end{eqnarray}
using \eqref{eq:expfast} and Lemma B.3 in the last inequality. Together with \eqref{eq:lowreach} this proves the lower bound in \eqref{eq:bounds_firststage}. For the upper bound, let $\overline{T}_0^N:=\inf\{i: \overline{K}_i=0\}.$ Note that 
\[\P_1((K_i)\mbox{ reaches }\varepsilon N)= \P((K_{i\wedge T_1^N})\mbox{ reaches }\varepsilon N)\]
and
\[ \P_1((K_{i\wedge T_1^N})\mbox{ reaches }\varepsilon N)=1-\P((K_{i\wedge T_1^N})\mbox{ dies out}).\]
 Thus we have
 \begin{eqnarray}
 1-\P_1((K_i) \mbox{ reaches }\varepsilon N)&\geq &\P_1((K_{i\wedge T_1^N}) \mbox{ dies out})\nonumber \\
 &\geq &\P_1((K_{i\wedge T_1^N})\mbox{ dies out}; (\overline{K}_i)\mbox{ dies out}; A; \overline{T}_0^N\leq \varrho_N^{-1-\delta})\nonumber \\
& =&\P_1((\overline{K}_i)\mbox{ dies out}; A; \overline{T}_0^N\leq \varrho_N^{-1-\delta})\nonumber \\
& \sim &\P_1((\overline{K}_i) \mbox{ dies out})\nonumber\\
& \sim & 1-\P_1((\overline{K}_i) \mbox{ reaches }\varepsilon N),
 \end{eqnarray}
where we have used \eqref{timeextGW} from the Appendix and Lemma \ref{Prop:criticalGW}. This implies the upper bound. \\
We are thus left with proving the last statement of the Lemma. Fix $\delta>0.$ We have
\begin{align}\label{eq:condchange}
\P_1(T_1^N>\varrho_N^{-1-\delta}\,|\, (K_i) \mbox{ reaches }\varepsilon N) 
=&\frac{\P_1(T_1^N>\varrho_N^{-1-\delta}, (K_i) \mbox{ reaches }\varepsilon N,  (\underline{K}_i) \mbox{ survives})}{\P_1((K_i)\mbox{ reaches }\varepsilon N)}\nonumber\\
&+\frac{\P_1(T_1^N>\varrho_N^{-1-\delta}, (K_i) \mbox{ reaches }\varepsilon N,  (\underline{K}_i) \mbox{ dies out})}{\P_1((K_i)\mbox{ reaches }\varepsilon N)}.
\end{align}
By \eqref{eq:asymeq} and Lemma \ref{Prop:criticalGW} we have for large enough $N$ the inequality $$\P_1((K_i)\mbox{ reaches }\varepsilon N)\geq \P_1( (\underline{K}_i) \mbox{ survives}),$$ and thus the first term on the right-hand side of \eqref{eq:condchange} can be bounded from above by
\begin{align}\label{eq:finaltime}
\P_1(T_1^N>\varrho_N^{-1-\delta}\,|\, (\underline{K}_i) \mbox{ survives})
\leq \,\,  &\P_1(T_1^N >\varrho_N^{-1-\delta}, A \,|\, (\underline{K}_i) \mbox{ survives})+\P_1(A^c\,|(\underline{K}_i) \mbox{ survives})\nonumber\\
\leq \,\,& \P_1(\underline{T}_1^N >\varrho_N^{-1-\delta} \,|\, (\underline{K}_i) \mbox{ survives})+\frac{\P_1(A^c)}{\P((\underline{K}_i) \mbox{ survives})}.
\end{align}
The first term on the right-hand side converges to $0$ due to
 Lemma \ref{lem:timefixGW}.  
By Lemma \ref{lem:GW} we have $\P_1((\underline{K}_i) \mbox{ survives})\sim c\varrho_N$, therefore by \eqref{eq:expfast} the second term on the right-hand side converges to $0$ as well. Thus we have shown that the first summand in \eqref{eq:condchange} converges to 0. To deal with the second term, we observe
\begin{align} \P_1(T_1^N>\varrho^{-1-\delta}_N&, (K_i) \mbox{ reaches }\varepsilon N,  (\underline{K}_i) \mbox{ dies out})\nonumber \\
\leq& \P_1((K_i) \mbox{ reaches }\varepsilon N,  (\underline{K}_i) \mbox{ dies out})\nonumber\\
= &\P_1((K_i) \mbox{ reaches }\varepsilon N,  (\underline{K}_i) \mbox{ dies out}, (\overline{K}_i) \mbox{ dies out})\nonumber\\
&+\P_1((K_i) \mbox{ reaches }\varepsilon N,  (\underline{K}_i) \mbox{ dies out}, (\overline{K}_i) \mbox{ survives})\nonumber\\
\leq &\P_1((K_i) \mbox{ reaches }\varepsilon N, (\overline{K}_i) \mbox{ dies out})+\P_1((\underline{K}_i) \mbox{ dies out},  (\overline{K}_i) \mbox{ survives})\nonumber\\
\leq &\P_1((K_i) \mbox{ reaches }\varepsilon N,  (\overline{K}_i) \mbox{ dies out}, \overline{K}_{\lfloor \varrho_N^{-1}\rfloor}>0)\nonumber\\
&+\P_1((K_i) \mbox{ reaches }\varepsilon N,  (\overline{K}_i) \mbox{ dies out}, \overline{K}_{\lfloor \varrho_N^{-1}\rfloor}=0)\nonumber\\
&+\P_1((\underline{K}_i) \mbox{ dies out},  (\overline{K}_i) \mbox{ survives}).
\end{align}

We have 
\begin{align*}
\P_1((K_i) \mbox{ reaches }\varepsilon N,  (\overline{K}_i) \mbox{ dies out}, \overline{K}_{\lfloor \varrho_N^{-1}\rfloor}>0)\leq \P((\overline{K}_i) \mbox{ dies out}, \overline{K}_{\lfloor \varrho_N^{-1}\rfloor}>0)
\end{align*}
which goes to 0 exponentially fast due to \eqref{timeextGW} in the Appendix, and using Lemma \ref{Prop:criticalGW} we get
\begin{align*}
\P_1((K_i) \mbox{ reaches }\varepsilon N,  (\overline{K}_i) \mbox{ dies out}, \overline{K}_{\lfloor \varrho_N^{-1}\rfloor}=0)\leq \P_1(A^c)
\end{align*}
which goes to 0 exponentially fast due to \eqref{eq:expfast}. Finally we have
\begin{align*}
\P_1((\underline{K}_i) \mbox{ dies out},  (\overline{K}_i) \mbox{ survives})=& \P_1((\underline{K}_i) \mbox{ dies out}\,|\, (\overline{K}_i) \mbox{ survives})\P_1((\overline{K}_i) \mbox{ survives})\\
\leq&\varepsilon (1+o(1))\P_1((\overline{K}_i) \mbox{ survives})\\
=&\frac{\varepsilon}{1-\varepsilon}(1+o(1))\P_1((\underline{K}_i) \mbox{ survives}),
\end{align*} see Corollary \ref{cor:probGW}. Thus the second summand in \eqref{eq:condchange} is bounded from above by $\frac{\varepsilon}{1-\varepsilon}(1+o(1)),$ and the claim follows.

\end{proof}

\begin{corollary}\label{1/2.3} Let $T_0^N:=\inf\{i: K_i=0\}.$ For $0<\varepsilon<1/\gamma\wedge 1/16$ there exists $N_\varepsilon^{(1)}$ such that for any $k\leq \varepsilon N,$
\be \label{eq:time1}
\P_k(T_1^N\wedge T_0^N>\varrho_N^{-1-\delta})\leq 1/2.
\ee
\end{corollary}

\begin{proof} 
Fix $k\leq \varepsilon N$. We have
\begin{eqnarray*}
\P_k(T_1^N\wedge T_0^N>\varrho_N^{-1-\delta})&=&\P_k(T_1^N>\varrho_N^{-1-\delta}|T_1^N\wedge T_0^N=T_1^N)\P_k(T_1^N\wedge T_0^N=T_1^N)\\
&\hspace{0.3cm} &+\;\,\P_k(T_0^N>\varrho_N^{-1-\delta}|T_1^N\wedge T_0^N=T_0^N)\P_k(T_1^N\wedge T_0^N=T_0^N).
\end{eqnarray*}
Due to \eqref{probtwoGW} we can see that all the steps leading to the last statement in Lemma \ref{lem:part3} hold if the processes are started in $k\leq \varepsilon N$ instead of 1.
Hence we have that for all $1\leq k\leq\varepsilon N$
\be \limsup_{N\to\infty}\P_k(T_1^N>\varrho_N^{-1-\delta}|T_1^N<\infty)\leq \frac{\varepsilon}{1-\varepsilon}.\ee
Moreover, if we stop $(K_i)$ with $K_0=k\leq \varepsilon N$ when the Markov chain is larger than $\varepsilon N$, then $(K_i)$ is an absorbing Markov chain  with absorbing states $0$ and any number larger than $\varepsilon N$. That implies $\P_k(T_1^N\wedge T_0^N<\infty )=1$.
Notice that under event $\{T_1^N\wedge T_0^N<\infty\}$, we have $\{T_1^N<\infty\}=\{T_1^N\wedge T_0^N=T_1^N\}$. 
Altogether we obtain
\be \label{eq:timetosurv} \limsup_{N\to\infty}\P_k(T_1^N>\varrho_N^{-1-\delta}|T_1^N\wedge T_0^N=T_1^N)\leq \frac{\varepsilon}{1-\varepsilon}(1+o(1)),\ee
which is smaller than $1/4$ for our choice of $\varepsilon.$ Therefore \eqref{eq:time1} holds for any $k\leq \varepsilon N$ such that 
$\P_k(T_0^N>\varrho_N^{-1-\delta}|T_1^N\wedge T_0^N=T_0^N)\leq 1/4.$ Assume therefore that $\P_k(T_0^N>\varrho_N^{-1-\delta}|T_1^N\wedge T_0^N=T_0^N)> 1/4.$ Due to Proposition \ref{GW1} and 
Lemma \ref{lem:timefixGW} we then have that  $\P_k(T_1^N\wedge T_0^N=T_0^N)\geq 1/4$ for $N$ large enough. For such $k$
\begin{align*}
\P_k(T_0^N>\varrho_N^{-1-\delta}\,|\,T_1^N\wedge T_0^N=T_0^N)&\leq \P_k(K_{\lfloor\varrho_N^{-1-\delta}\rfloor}>0, A\,|\,T_1^N\wedge T_0^N=T_0^N)+\P_k(A^c\,|\,T_1^N\wedge T_0^N=T_0^N)\\
&\leq 4\P_k(\overline{K}_{\lfloor \varrho_N^{-1-\delta}\rfloor}>0,(\underline{K}_{i})_{i\in \N} \mbox{ dies out})+4\P_k(A^c).\\
\end{align*}
Equation \eqref{probtwoGW2} implies
\[4\P_k(\overline{K}_{\lfloor \varrho_N^{-1-\delta}\rfloor}>0,(\underline{K}_{i})_{i\in \N} \mbox{ dies out})\leq 4 \P_k(\overline{K}_{\lfloor \varrho_N^{-1-\delta}\rfloor}>0,(\overline{K}_{i})_{i\in \N} \mbox{ dies out})+4\varepsilon(1+o(1)).\]
By \eqref{eq:expfast}, $\P_k(A^c)$ goes to 0 exponentially fast, and $\P_k(\overline{K}_{\varrho_N^{-1-\delta}}>0\,|\,(\overline{K}_{i})_{i\in\N_0}\mbox{ dies out})$ goes to 0 by \eqref{timeextGW}. Thus if $\varepsilon<1/16$ the right-hand side of the above inequality is bounded above by $1/4,$ and we have completed the proof.
\end{proof}

\subsection{Second stage of the sweep}\label{subsect:secondstage}

\begin{lemma}
\label{lem:fix_epsilon}
For $\varepsilon\in (0,1/2)$ let $1-\varepsilon'\in(\varepsilon,1).$ Then we have for any $k\geq \varepsilon N$
  \[\lim_{N\to\infty}\P_k(\exists i: K_i\geq \lfloor(1-\varepsilon') N \rfloor)=1.\]
Moreover, $\lim_{N\to\infty}\P(T_2^N-T_1^N>\varrho_N^{-1-\delta})=0$ for any $\delta>0,$ where we recall $T^N_{2}=\inf\{i: K_i\geq (1-\varepsilon) N\}.$ 
\end{lemma}

\begin{proof}We use an ODE approximation. Recall that $K_i$ denotes the number of mutants at the beginning of day $i.$ Let  $x\in[\varepsilon, 1).$ 
From Corollary \ref{cor:sigma}, we obtain that the expected number of offspring at the end of day $i$ of a \textit{single} mutant, given that there are $\lfloor xN \rfloor$ mutants at the beginning of the day, is given by $e^{(r+\varrho_N)\sigma_{\lfloor x N\rfloor}}$. 
Using Corollary \ref{cor:sigma}, we obtain
\be \label{eq:exp_K} \E[K_i\,|\, K_{i-1}=\lfloor x N\rfloor]=\frac{\lfloor x N\rfloor}{\gamma}\big(e^{(r+\varrho_N)\sigma_{\lfloor x N\rfloor}}\big)
=\lfloor x N\rfloor \big(1+\varrho_N\frac{\log \gamma}{r}(1-xN)+O(\varrho_N^2)\big).\ee
From Lemma \ref{xgeo} b) and Corollary \ref{cor:sigma} we see that
there exists $c=c(\gamma, r)<\infty$ such that 
$$
{\var(K_i\,|\, K_{i-1}=k)\leq c N,  k=1,2,\ldots,N}.
$$ 
For $f\in C^2[0,1]$ we define the rescaled discrete generator of $(K_i)_{i\in\N_0}$ 
\[ A_Nf(\frac{k}{N})=\varrho_N^{-1}\E[f(K_i/{N})-f(k/N)\,|\,K_{i-1}=k ],\quad x\in [0,1].\]
Using Taylor approximation on $f$ we infer that, for some $y\in[0,1]$,
\[ A_Nf(\frac{k}{N})=\varrho_N^{-1}\Big(\E\big[(\frac{K_i}{{N}}-\frac{k}{N})f'(\frac{k}{N})+{\frac{1}{2}}\big(\frac{K_i}{{N}}-\frac{k}{N}\big)^2f''(y)\,\big|\,K_{i-1}=k\big]\Big).\]
We have,
\be \begin{split}   
     \E_{k}\big[\big(\frac{K_1}{{N}}-\frac{k}{N}\big)^2\big]=& \frac{1}{{N^{2}}}\E_{k }[K_1^2-\E_{k}[K_1]^2]+\frac{1}{{N^{2}}}\E_{k }[K_1]^2-2\frac{k}{N^2}\E_{k}[K_1]+(\frac{k}{N})^2\\
     =& \frac{1}{{N^{2}}}\text{var}_{k}(K_1)+\big(\E_{k }[K_1]/{N}-x\big)^2\\
     \leq & {\frac{{c}}{N}}+O(\varrho_N^2),
    \end{split}
\ee
where $|O(\varrho_N^2)|/\varrho_N^2$ is bounded uniformly in $N$ and $k$. Hence recalling $\varrho_N^{-1}N^{-\alpha}\to 0$ for $\alpha>b$ and the continuity of $f''$ on $[0,1]$, we obtain the following convergence which is uniform in $k$ and $y$:
\[\sup_{y\in[0,1], k=0,1,\ldots,N}|\varrho_N^{-1}\Big(\E_{k }\big[(\frac{K_1}{N}-\frac{k}{N}\big)^2\big]f''(y)\Big)| \to 0, N\to\infty.
\]

Since $\E_{k }[\frac{K_1}{N}-\frac{k}{N}]=\frac{k}{\gamma}e^{(r+\varrho_N)\sigma_k}-\frac{k}{N},$ one can apply Corollary \ref{cor:sigma}. Together with the above display, we obtain
\be \nonumber
\sup_{k=0,1,\ldots,N}|A_Nf(\frac{k}{N})-\frac{k}{N}(1-\frac{k}{N})f'(\frac{k}{N})|\to 0, N\to\infty.
\ee
Applying Theorem 1.6.5 and Theorem 4.2.6 of \cite{ethierkurtz} we infer that for every $x\in[0, 1]$, the sequence of processes $(\frac{1}{N}K_{\lfloor \varrho_N^{-1}t\rfloor})_{t\ge 0}$, $N=1,2,\ldots$, $K_0=\lfloor xN \rfloor$ converges locally uniformly in distribution to the deterministic (increasing) function $g(t)$ which is defined by the initial value problem
\[g'(t)=g(t)(1-g(t))\frac{\log\gamma}{r},\quad g(0)= x\in[0, 1].\]
Now choose $t_*$ such that $g(t_*) > 1-\varepsilon'$, provided $g(0) = \varepsilon>0$. This implies
\be \label{conv_1}\lim_{N\to\infty}\P(K_{\lfloor \varrho_N^{-1}t_*\rfloor}\geq \lfloor(1-\varepsilon')N\rfloor|K_0\geq \lfloor \varepsilon N\rfloor)=1, \ee
and \it a fortiori, $\lim_{N\to\infty}\P(T_2^N-T_1^N>\varrho_N^{-1-\delta})=0$ for any positive $\delta.$
\end{proof}

\begin{corollary}\label{1/2.1}
For any $\varepsilon\in(0,1/2)$, there exist $N_\varepsilon^{(2)}\in \N$ such that for every $N>N_\varepsilon^{(2)},$ for every $k\geq \varepsilon N,$
\[\P_k(\exists i\leq \varrho_N^{1-\delta}: K_i\geq \lfloor(1-\varepsilon) N \rfloor)\geq 1/2.\]
\end{corollary}
\begin{proof}
The proof follows immediately from \eqref{conv_1}.
\end{proof}

\subsection{Third stage of the sweep}\label{subsect:thirsdstage}
For the last stage of the sweep, after the number of mutants has reached at least $(1-\varepsilon)N,$ we use a Galton-Watson coupling similar in spirit to the coupling at the first stage. The difference is that this time we will be working with the process of wild type individuals rather than the mutants.
Fix again $\alpha\in (b,1/2).$ Let $Q_i:=N-K_i$ be the number of wild-type individuals at the beginning of day $i.$ We  proceed similarly as in Section \ref{subsect:GW-constr} to define approximating Galton-Watson processes $(\underline{Q}_i)_{i\in\N_0}$ and $(\overline{Q}_i)_{i\in\N_0},$ for $i\in\N$ constructing $\underline{Q}_i$ and $\overline{Q}_i$ recursively from the same Yule forest as $Q_i:$ Recall that the wild type individuals reproduce at rate $r.$ Assume that $\underline{Q}_{i-1}$ and $\overline{Q}_{i-1}$ are constructed, and start independent Yule trees growing at rate $r$ for each individual as we did in Section \ref{subsect:GW-constr} to construct $\overline{K}_i$ and $\underline{K}_i.$ Assume $Q_{i-1}=q \in (0, \varepsilon N).$ Grow the Yule trees until time $\sigma_{\lceil (1-2\varepsilon)N \rceil}$ and distinguish the individuals according to whether they were born before $\sigma_N,$ before $\sigma_{N-q},$ or before $\sigma_{\lceil (1-2\varepsilon )N \rceil}.$ Taking the time of birth into 
consideration, the individuals born before $\sigma_N$ will be sampled independently with probability $\gamma^{-1}-N^{-\alpha}$ to form $\underline{Q}_i,$ born before $\sigma_{N-q}$ will be chosen according to \eqref{def:Xtilde} to form $Q_i$, and those before $\sigma_{\lceil (1-2\varepsilon)N \rceil}$ with probability $\gamma^{-1}+N^{-\alpha}$ to form $\overline{Q}_i.$ \\

It is clear that Lemma \ref{prop1} and Corollary \ref{GWbounds} still hold, and thus we can prove the equivalent to Proposition \ref{GW1}. Define
\[T_w^N(m):=\inf\{i:Q_i>m\varepsilon N \text{ or }Q_i=0\}, \;m\geq 1.\]
\begin{lemma}\label{GW2}
Let $\alpha\in(b,1/2).$ Let $m\geq 1,$ and $0<\varepsilon<1/(m\gamma).$ Assume $\underline{Q}_0=Q_0=\overline{Q}_0\leq \varepsilon N.$ Then there exists $c$ large enough such that for $N$ large enough,
\be
\P\Big(\overline{Q}_{\min\{i,T_w^N(m)\}}\geq Q_{\min\{i,T_w^N(m)\}} \geq \underline{Q}_{\min\{i,T_w^N(m)\}},\forall i\leq g \Big)\geq (1-2e^{-c N})^g \quad \mbox{for all }g\in\N_0. \nonumber
\ee
for some constant $c$ independent of $N.$
\end{lemma}  

\begin{proof}
This follows from a straightforward adaptation of the proof of Proposition \ref{GW1}, since the condition $\varepsilon\leq 1/(m\gamma)$ allows us to prove the analog of Lemma \ref{lem:one-step}, observing that the definition of $T_w^N(m)$ ensures that we stop the procedure if $Q_i$ reaches $m\varepsilon N$ individuals (and not $\varepsilon N$ as in Proposition~\ref{GW1}). 
\end{proof}
 
We have the alternative description corresponding to Proposition \ref{def:GW-processes}: $(\overline{Q}_i)_{i\in\N_0}$ is the Galton-Watson process whose offspring distribution is mixed binomial with parameters $\overline{W}$ and $\frac{1}{\gamma}+N^{-\alpha},$ where $\overline{W}$ is geometric with parameter $e^{-r\sigma_{\lceil (1-2\varepsilon)N \rceil}}.$ Similarly, $(\underline{Q}_i)_{i\in\N_0}$ is the Galton-Watson process whose offspring distribution is mixed binomial with parameters $\underline{W}$ and $\frac{1}{\gamma}-N^{-\alpha},$ where $\underline{W}$ is geometric with parameter $e^{-r\sigma_{N}}.$
From this we obtain the analogue of Lemma \ref{lem:expvarGW}.

\begin{lemma}
For $(\underline{Q}_i)_{i\in\N_0}$ and $(\overline{Q}_i)_{i\in\N_0}$ defined above there exist $\overline{c}, \underline{c}$ independent of $N$ such that for $N$ large enough,
\begin{equation}\label{Qestimate}
\E_1[\overline{Q}_1]=1-\overline{c}\varrho_N+o(\varrho_N)\, \text{ and }\,\E_1[\underline{Q}_1]=1-\underline{c}\varrho_N+o(\varrho_N)
\end{equation} 
\end{lemma}
\begin{proof} By construction, and from Corollary \ref{cor:sigma}
\begin{eqnarray*}
\E_1[\underline{Q}_1]&=&
(1/\gamma-N^{-\alpha})\E[\underline{W}]=(1/\gamma-N^{-\alpha})e^{r\sigma_{N}}\\
&=&(1/\gamma-N^{-\alpha})(\gamma-\varrho_N\frac{\log \gamma}{r})+o(\varrho_N)\\
&=&1-\frac{\log \gamma}{\gamma r} \varrho_N+o(\varrho_N),
\end{eqnarray*} 
where the last equality follows from the fact that our assumptions imply that $ N^{-\alpha}=o(\varrho_N)$. This is the first assertion in \eqref{Qestimate}.  In the same way we obtain
$
\E_1[\overline{Q}_1]=1-\overline{c} \varrho_N+o(\varrho_N)
$, 
for some positive constant  $\underline{c}$ independent of $N$.
\end{proof}

\begin{lemma}\label{almostsurefixation}
Let $m\geq 1$ and $0<\varepsilon<1/(m\gamma)$. For any $k\geq (1-\varepsilon)N$,
\[\limsup_{N\to\infty}\P_k(\tau_{\rm fix}^N>\varrho_N^{-1-\delta})\leq 2/m\]
for any $\delta>0.$ In particular, $\P_k(\exists i: K_i=N)\geq 1-2/m.$
\end{lemma}

\begin{proof}
Under  $\P_k$ we have by assumption that $K_0=k\geq (1-\varepsilon )N,$ and thus $Q_0=N-k\leq \varepsilon N.$ 
We consider $(\overline Q_i)_{i\in\N_0},(\underline{Q}_i)_{i\in\N_0}$ as constructed at the beginning of this section, with $\alpha\in (b,1/2).$ Let 
\[A:=A(\gamma,\alpha,\varepsilon,N, m):=\Big\{\overline{Q}_{\min\{i,T_w^N(m)\}}\geq Q_{\min\{i,T_w^N(m)\}} \geq \underline{Q}_{\min\{i,T_w^N(m)\}},\forall i\leq \varrho_N^{-1-\delta}\Big\}.\]
Then Lemma \ref{GW2} shows
$$\P(A)\to1 \mbox{ as } N\to\infty.$$
Note that 
\[\E_k[\overline{Q}_{\lfloor \varrho_N^{-1-\delta}\rfloor}]\sim(N-k)(1-\bar c\varrho_N)^{\varrho_N^{-(1+\delta)}}\leq (N-k)e^{-\bar{c}\varrho_N^{-\delta}}\leq \varepsilon Ne^{-\bar{c}\varrho_N^{-\delta}}\rightarrow 0\] 
as $N\to\infty.$ Consequently, since on the event $\{T_{\omega}^N(m)>\varrho_N^{-1-\delta}\}\cap A$ we have  $Q_{\lfloor \varrho_N^{-1-\delta}\rfloor }\geq 1,$
\begin{align}\label{ts}\P_k(T_w^N(m)>\varrho_N^{-1-\delta})&\leq  \P_k(T_w^N(m)>\varrho_N^{-1-\delta}, A)+\P_k(A^c)\leq \E_k[Q_{\lfloor \varrho_N^{-1-\delta}\rfloor }\ind_{\{T_w^N(m)>\varrho_N^{-1-\delta}\}}\ind_A]+\P_k(A^c)\nonumber\\
&\leq \E_k[\overline{Q}_{\lfloor \varrho_N^{-1-\delta}\rfloor }]+\P_k(A^c)\to 0 \mbox{ as } N\to\infty.\nonumber \end{align}
{Since
\begin{align*}
\P_k(\tau_{\rm{fix}}^N>\varrho_N^{-1-\delta})=&\P_k(\tau_{\rm{fix}}^N>\varrho_N^{-1-\delta}, T_w^N(m)>\varrho_N^{-1-\delta})+\P_k(\tau_{\rm{fix}}^N>\varrho_N^{-1-\delta}, T_w^N(m)\leq \varrho_N^{-1-\delta})\\
\leq &\P_k(T_w^N(m)>\varrho_N^{-1-\delta})+\P_k(Q_{T_w^N(m)}\geq \varepsilon mN),
\end{align*}
we are left with proving 
\be \label{eq:Qsurv}\limsup_{N\to\infty}\P_k(Q_{T_w^N(m)}\geq \varepsilon mN)\leq 2/m.\ee}
Let $\kappa$ be the first time when $(\overline{Q}_i)_{i\geq 0}$ is not less than $\varepsilon mN$ or equal to $0$. Note that under $A\cap \{T_{w}^N(m)\leq \varrho_N^{-1-\delta}\}$, if $Q_{T_w^N(k)}\geq \varepsilon mN$, then necessarily, $\overline Q_{T_w^N(m)}\geq \varepsilon mN$. So in conclusion:
\begin{equation}\label{compare}\P_k(Q_{T_{w}^N(m)}\geq \varepsilon mN, A, T_{w}^N(m)\leq \varrho_N^{-1-\delta})\leq \P_k(\overline Q_{\kappa}\geq \varepsilon mN, A, T_{w}^N(m)\leq \varrho_N^{-1-\delta}).\end{equation}
Notice that $(\overline{Q}_i)_{i\geq 0}$ is, as a sub-critical Galton-Watson process, a supermartingale. Then $(\overline{Q}_i\wedge \varepsilon mN)_{i\geq 0}$ is a bounded supermartingale and, for any time strictly before $\kappa$, these two supermartingales are the same. Now we have 
$$\varepsilon N\geq \E_{k}[\overline Q_0]=\E_k[Q_0\wedge \varepsilon mN]\geq \E_k[\overline Q_{\kappa}\wedge \varepsilon m N]=\P_k(\overline Q_{\kappa}\geq \varepsilon mN)\varepsilon mN.$$
So
$$\P_k(\overline Q_{\kappa}\geq \varepsilon mN)\leq 1/m.$$
Therefore using (\ref{compare}) we have for $N$ large enough
$$\P_k(Q_{T_{w}^N(m)}\geq \varepsilon mN)\leq \P_k(\overline{Q}_{\kappa}\geq \varepsilon mN)+\P_k(T_{w}^N(m)>\varrho_N^{-1-\delta})+\P(A^c)\leq 2/m.$$
This implies \eqref{eq:Qsurv}, and moreover $\P_k(\exists i: K_i=N)=\P_k(Q_{T_{w}^N(k)}=0)\geq 1-2/m.$

\end{proof}

This result will be useful in the following simple form:

{\begin{corollary}\label{1/2.2}
For every $0<\varepsilon<1/(4\gamma)$ there exist $N_\varepsilon^{(3)}\in \N$ such that for all $N\geq N^{(3)}_{\varepsilon},\delta>0$ and $k\geq  (1-\varepsilon)N$
\[\P_k(\tau_{\rm fix}^N>\varrho_N^{-1-\delta})\leq 1/2.\]
\end{corollary}
\begin{proof}
Take $m\geq 4$ in Lemma \ref{almostsurefixation}.
\end{proof}}

\subsection{Proof of Theorem \ref{thm:fixation}}\label{subsect:proof} We are now finally able to prove Theorem \ref{thm:fixation}. Let $m\geq 4$ and $0<\varepsilon<1/(m\gamma)\wedge 1/16$.
By Lemma \ref{lem:part3} we have
\[\pi_N=\P_1(\exists i: K_i=N)\leq \P(K_i\mbox{ reaches } \varepsilon N)\leq \frac{\gamma\log\gamma}{\gamma-1}\frac{\varrho_N}{r}(1+o(1)).\]
Further, observe that for $1\leq k\leq k'\leq l\leq N,$ by definition of the model,
\[\P_k(K_1\geq l)\leq \P_{k'}(K_1\geq l)\]
and therefore by induction $\P_k(K_i\geq l)\leq \P_{k'}(K_i\geq l), i\in\N.$ Thus 
\[\P_k((K_i)\mbox{ reaches }l)\leq \P_{k'}((K_i)\mbox{ reaches }l).\]
Therefore, for every $\varepsilon \in (0,1/(m\gamma)\wedge 1/16),$ by the strong Markov property and  Lemma \ref{lem:part3}, 
\begin{align*}
\pi_N\geq&\P_{\lfloor \varepsilon N\rfloor}(\exists i: K_i=N)\cdot \P_1(K_i \mbox{ reaches }\varepsilon N)\\
\geq & \P_{\lfloor \varepsilon N\rfloor}(\exists i: K_i=N)\cdot \frac{\gamma\log \gamma}{\gamma-1}\frac{\varrho_N}{r}(1-\varepsilon)(1+o(1)).
\end{align*}
From Lemmas \ref{almostsurefixation} and \ref{lem:fix_epsilon} we obtain $\liminf_{N\to\infty}\P_{\lfloor \varepsilon N\rfloor}(\exists i: K_i=N)\geq 1-2/m$ for any $m\geq 2$. Thus 
\[(1-\varepsilon)(1-2/m)\leq \liminf_{N\to\infty} \frac{\gamma-1}{\gamma\log \gamma}\frac{r}{\varrho_N}\pi_N\leq \limsup_{N\to\infty} \frac{\gamma-1}{\gamma\log \gamma}\frac{r}{\varrho_N}\pi_N\leq 1.\]
Sending $m\to\infty$ (and $\varepsilon\to 0)$ gives \eqref{probfix}. \\

Now we will prove that $\P_1(\tau^N>\varrho_N^{-1-2\delta})\leq (7/8)^{\varrho_N^{-\delta}}.$
Let $N_\varepsilon=\sup\{N_\varepsilon^{(1)}, N_\varepsilon^{(2)}, N_\varepsilon^{(3)}\}$ where $N_\varepsilon^{(1)}, N_\varepsilon^{(2)}, N_\varepsilon^{(3)}$ can be found respectively in Corollary \ref{1/2.3}, \ref{1/2.1} and \ref{1/2.2}. Using the three corollaries and the strong Markov property of the process $(K_i)_{i\in\N_0}$ we know that for all $N>N_\varepsilon$, and for any $k\in \{1,2,...,N\}$
\be \label{eq:s_Nonehalf}
\P_k(\tau^N\leq 3\varrho_N^{-1-\delta})\geq (1/2)^3.
\ee
Using the Markov property at time $\lceil 3 \varrho_N^{-1-\delta}\rceil$, we see that for any $n\in \N$
\begin{align*}
\P_1(\tau^N>3n\varrho_N^{-1-\delta} )&\leq \P_1(\tau^N>\lceil 3 \varrho_N^{-1-\delta}\rceil)\sum_{k=1}^{N-1} \P_k(\tau^N>3(n-1) \varrho_N^{-1-\delta})\P_1(K_{\lceil \varrho_N^{-1-\delta}\rceil }=k)\\
&\leq (1-(1/2)^{3}) \sum_{k=1}^{N-1} \P_k(\tau^N>3(n-1)\varrho_N^{-1-\delta})\P_1(K_{\lceil \varrho_N^{-1-\delta}\rceil }=k).
\end{align*}
Thus, proceeding iteratively, and using the fact that \eqref{eq:s_Nonehalf} is uniform in $k\in\{1,...,N-1\},$ we obtain
\[\P_1(\tau^N>3n\varrho_N^{-1-\delta})\leq (1-(1/2)^3)^n.\]
In particular, choosing  $n=\lceil \varrho_N^{-\delta}\rceil $ we obtain for $\delta>0$
$$
\P_1(\tau^N>\varrho_N^{-1-3\delta})\P_1(\tau^N>3\varrho_N^{-1-2\delta})\leq (7/8)^{\varrho_N^{-\delta}}.
$$
\hfill $\Box$

\subsection{Proof of Proposition \ref{noclonalint}}\label{subsect:proofnoclonal}
Due to Theorem \ref{thm:fixation}, and due to the Assumption that the mutations arrive independently of each other at geometric times with parameter $\mu_N,$ we have that for any $\delta'>0$
\begin{align*} 
\P(m_N<\tau^N)\leq & 1-\P(m_N>\varrho_N^{-1-\delta'}\,|\, \tau^N<\varrho_N^{-1-\delta'})\P(\tau^N<\varrho_N^{-1-\delta'})\\
\leq &1-(1-\mu_N)^{\lfloor \varrho_N^{-1-\delta'}\rfloor }(1-(7/8)^{\lfloor \varrho_N^{-\delta'/3}\rfloor }).
\end{align*}
 Now the Bernoulli inequality yields
\begin{align*}
\P(m_N<\tau^N)&\leq 1-(1-\mu_N \lfloor \varrho_N^{-1-\delta'}\rfloor )(1-(7/8)^{\lfloor \varrho_N^{-\delta'/3}\rfloor })\\
&=\mu_N \lfloor \varrho_N^{-1-\delta'}\rfloor +(7/8)^{\lfloor \varrho_N^{-\delta'/3}\rfloor }-\mu_N \lfloor \varrho_N^{-1-\delta'}\rfloor (7/8)^{\lfloor \varrho_N^{-\delta'/3}\rfloor }.
\end{align*}
From this we obtain 
\[\P(m_N<\tau^N)\leq \mu_N \varrho_N^{-1-\delta}\]
for any $\delta>\delta',$ provided $N$ is large enough.
This proves the first claim. Now, let $E_j$ be the event that there is no clonal interference until the day that the $j$-th successful mutation starts. Observe that $\P(E_1)$ is given by the probability that any unsuccessful mutation started before the first successful one has disappeared before the next mutation (successful or unsuccessful) starts. By the first part of this theorem, for any given mutation this is the case with probability $\P(m_N\geq \tau^N)\geq 1-\mu_N \varrho_N^{-1-\delta},$ for $\delta>0.$ Denote by $L$ the number of mutations until the first successful one. Since the mutations arrive independently of each other, we see by induction that for $l\in\N_0$
\[\P(\mbox{ no clonal interference in the first $l$ mutations}\,|\, L=l+1)\geq (1-\mu_N\varrho_N^{-1-\delta})^l.\]
By Theorem \ref{thm:fixation}, $L$ is (asymptotically) geometric with success parameter $C(\gamma)\varrho_N/r_0.$ Thus summing over all possible values of $L$ we obtain
by Theorem \ref{thm:fixation} and the first part of this proof, for $\delta>0,$
\begin{eqnarray*}
\P(E_1)&\geq &\sum_{l=0}^\infty\P(L=l+1)(1-\mu_N\varrho_N^{-1-\delta})^l\\
&\geq &\sum_{l=0}^\infty (1-\frac{C(\gamma)\varrho_N}{r_0})^{l}\frac{C(\gamma)\varrho_N}{r_0}(1-\mu_N\varrho_N^{-1-\delta})^l\\
&= &\frac{C(\gamma)\varrho_N}{r_0}\sum_{i=0}^\infty (1-\frac{C(\gamma)\varrho_N}{r_0}-\mu_N\varrho_N^{-1-\delta}+3\frac{C(\gamma)}{r_0}\varrho_N^{-\delta})^{l}\\
&=&\frac{C(\gamma)\varrho_N}{r_0}\frac{1}{C(\gamma)\varrho_Nr_0^{-1}+\mu_N\varrho_N^{-1-\delta}-C(\gamma)r_0^{-1}\mu_N\varrho_N^{-\delta}}\\
&\geq &1-\mu_N\varrho_N^{-2-\delta''}+o(\mu_N\varrho_N^{-2-\delta''})
\end{eqnarray*}
for $N$ large enough and $\delta''>\delta.$
Fix $n\in\N.$ Similar to the previous calculation, for $j\leq n\varrho_N^{-1},$ we have $\P(E_{j+1}|E_j)\geq 1-\mu_N\varrho_N^{-2-\delta''}/3+o(\mu_N\varrho_N^{-2-\delta''}).$ 
Proceeding iteratively one thus observes that for any fixed $n\in \N$
\begin{eqnarray}\label{eq:En}
\P(E_{\lfloor \varrho_N^{-1}n\rfloor })&\geq &(1-\mu_N\varrho_N^{-2-\delta''}+o(\mu_N\varrho_N^{-2-\delta}))^{\lfloor n\varrho_N^{-1}\rfloor }\nonumber\\
&\geq &1-n\mu_N\varrho_N^{-3-3\delta''}(1+o(1)).
\end{eqnarray}
By Assumption A iii) this tends to 1 for $\delta''>0$ small enough. Let $I_n$ be the day at which the $\lfloor \varrho_N^{-1}n\rfloor $-th successful mutation starts.  We can write 
\[I_n=\sum_{j=1}^n I^{(j)},\]
if $I^{(j)}$ denotes the time between the fixation of the $j-1$th and the initiation of the $j$th successful mutation (and $I^{(1)}=I_1$). {Let $L^{(j)}$ denote the number of unsuccessful mutations that happen during time $I^{(j)}.$ The success probability of a mutation that happens during $I^{(j)}$ is according to Theorem \ref{thm:fixation} given by $C(\gamma)\frac{\varrho_N}{r_0+(j-1)\varrho_N}.$ Therefore, conditional on $E_j,$ $L^{(j)}$ is geometrically distributed with success parameter $C(\gamma)\frac{\varrho_N}{r_0+(j-1)\varrho_N}.$ Moreover, conditional on $E_j,$ the time between two of the $L^{(j)}$ unsuccessful mutations is stochastically larger than a geometric random variable with parameter $\mu_N,$ since this is the rate at which mutations arrive, and the geometric distribution is memoryless. Thus we see that the time $I^{(j)}$ is stochastically larger than a geometric random variable with parameter $\frac{C(\gamma)\mu_N\varrho_N}{r_0+(j-1)\varrho_N}$ and a fortiori stochastically larger than 
$G^N_j$, if}
$(G^N_j)_{j\in\N_0}$ is a sequence of independent geometric random variables with parameter $C(\gamma)\mu_N\varrho_N/r_0$. Thus conditionally on $E_{\lfloor n\varrho_N^{-1}\rfloor},$ stochastically $I_n\geq \sum_{j=1}^{\lfloor \varrho_N^{-1}n\rfloor }G_j^N$. Let $n=\lceil 2Tr_0/C(\gamma)\rceil $.
 Then
\begin{eqnarray*}
\lim_{N\to\infty} \P(\mbox{no clonal interference until } \varrho_N^{-2}\mu_N^{-1}T)&\geq &\P(E_{\lfloor \varrho_N^{-1}n\rfloor },\;I_n>\lceil \varrho_N^{-2}\mu_N^{-1}T\rceil)\\
&=&\P(E_{\lfloor \varrho_N^{-1}n\rfloor })\P(I_n>\lceil \varrho_N^{-2}\mu_N^{-1}T\rceil |E_{\lfloor \varrho_N^{-1}n\rfloor })\\
&\geq &\P(E_{\lfloor \varrho_N^{-1}n\rfloor })(1-2\P(\sum_{j=1}^{\lfloor \varrho_N^{-1}n\rfloor }G_j^N<\lfloor \varrho_N^{-2}\mu_N^{-1}T\rfloor ))
\end{eqnarray*}
By Cram\'er's large deviation principle the second factor tends to 1. Thus the statement follows from~\eqref{eq:En}.
\hfill{$\Box$}

\subsection{Proof of Theorem \ref{thm:mutations}}\label{proof:mut}
Denote by $D_i$ the event that there is no clonal interference up to day $i,$ that is, any mutation that starts until or including day $i$ happens in a homogeneous population. Define
\[\tilde{H}_i:=H_i1_{D_i}-\infty 1_{D_i^c}.\]
Then we have for any $T>0$ that the  two processes $(H_i)_{1\leq i\leq \varrho^{-2}_N\mu_N^{-1}T}$ and $(\tilde{H}_i)_{1\leq i\leq \varrho^{-2}_N\mu_N^{-1}T}$ coincide on the event $(D_{\lceil\varrho_N^{-2}\mu_N^{-1}T\rceil}^c)$, whose probability converges to $0$
as $N\to\infty$, by Proposition \ref{noclonalint}. Thus it is sufficient to show that $(\tilde{H}_{\lfloor t \varrho^{-1}_N\mu_N^{-1}\rfloor})_{0\le t\le T}$ converges in distribution to $(M(C(\gamma)t/r_0))_{0\le t\le T}$ w. r. to the Skorokhod topology, cf. Theorem 3.3.1 in \cite{ethierkurtz}. This will be achieved by a standard generator calculation.\\
The process $(H_i)_{i\in\N_0}$ is a Markov chain on $\N_0\cup\{-\infty\}$ with the following transition probabilities: If $n\geq 0,$ then
\begin{align*}
\P(\tilde{H}_{i+1}=n+1\,|\, \tilde{H}_i=n)&=\frac{C(\gamma)\mu_N\varrho_N}{r_0+n\varrho_N}\P(D_{i+1}\,|\, D_i),\\
\P(\tilde{H}_{i+1}=n\,|\, \tilde{H}_i=n)&=\Big(1-\frac{C(\gamma)\mu_N\varrho_N}{r_0+n\varrho_N}\Big)\P(D_{i+1}\,|\, D_i),\\
\P(\tilde{H}_{i+1}=-\infty\,|\, \tilde{H}_i=n)&=\P(D_{i+1}^c\,|\, D_i),
\end{align*}
and
\[\P(\tilde{H}_{i+1}=-\infty\,|\, \tilde{H}_i=-\infty)=1.\]
Observe first that for any $\delta>0$ we have
\be \label{eq:noclon_day}\P(D_{i+1}^c\,|\, D_i)\leq  \mu_N^2\varrho_N^{-1-\delta}.\ee
This follows since conditional on the event $D_i,$ the event $D_{i+1}^c$ can only happen if at day $i+1$ a new mutation happens, and interferes with the previous mutation. The probability that a new mutation happens is given by $\mu_N,$ and the probability of interference of a pair of mutations is $\P(m_N<\tau^N).$ Thus \eqref{eq:noclon_day} follows from Proposition \ref{noclonalint}.\\
For bounded functions $g$ on $\N_0\cup\{-\infty\},$ the discrete generator of $(\tilde{H}_i)_{i\in\N_0}$ on the time scale $i=\varrho_N^{-1}\mu_N^{-1}t$ is given by (cf. Theorem 1.6.5 of \cite{ethierkurtz})
\begin{align*}
B_Ng(n):=&\frac{1}{\varrho_N\mu_N}\E\big[ g(\tilde{H}_{i+1})-g(n)\,\big|\, \tilde{H}_i=n\big]\\
=& \frac{1}{\varrho_N\mu_N}\Big(\frac{C(\gamma)\mu_N\varrho_N}{r_0+n\varrho_N}\P(D_{i+1}\,|\, D_i)(g(n+1)-g(n))+\P(D^c_{i+1}\,|\, D_i)(g(-\infty)-g(n))\Big)\\
=&\frac{C(\gamma)}{r_0+n\varrho_N}\P(D_{i+1}\,|\, D_i)(g(n+1)-g(n))+\frac{\P(D^c_{i+1}\,|\, D_i)}{\varrho_N\mu_N}(g(-\infty)-g(n)).
\end{align*}
Due to \eqref{eq:noclon_day} and Assumption A iii), the r.h.s. converges as $N\to\infty$ to
\[\frac{C(\gamma)}{r_0}(g(n+1)-g(n)),\]
which is the generator of the Poisson process $(M(C(\gamma)t/r_0))_{t\geq 0}.$ By Theorem 4.2.6 of \cite{ethierkurtz} this implies convergence of the corresponding processes.
\hfill $\Box$

\subsection{Convergence of the fitness process}\label{proof:adapt}
\begin{proof}[Proof of Theorem \ref{thm:adapt}] We proceed analogously to the proof of Theorem \ref{thm:mutations}. Define
\[\tilde{\Phi}_i:=1+\frac{\varrho_N}{r_0}\tilde{H_i},\]
and recall $\Phi_i=1+\frac{\varrho_N}{r_0}H_i.$ As above, observe that the two processes
 $(\Phi_i)_{1\leq i\leq \varrho^{-2}_N\mu_N^{-1}T}$ and $(\tilde{\Phi}_i)_{1\leq i\leq \varrho^{-2}_N\mu_N^{-1}T}$ coincide on the event  $D_{\lceil\varrho_N^{-2}\mu_N^{-1}T\rceil}^c$, whose probability converges to $0$ as
as $N\to\infty$, and that $(\tilde{\Phi}_i)_{i\in\N_0}$ is a Markov chain with transition probabilities
\begin{align*}
\P(\tilde{\Phi}_{i+1}=x+\frac{\varrho_N}{r_0}\,|\, \tilde{\Phi}_i=x)&=\frac{C(\gamma)\mu_N\varrho_N}{xr_0}\P(D_{i+1}\,|\, D_i),\\
\P(\tilde{\Phi}_{i+1}=x\,|\, \tilde{\Phi}_{i}=x)&=\Big(1-\frac{C(\gamma)\mu_N\varrho_N}{xr_0}\Big)\P(D_{i+1}\,|\, D_i),\\
\P(\tilde{\Phi}_{i+1}=-\infty\,|\, \tilde{\Phi}_i=x)&=\P(D_{i+1}^c\,|\, D_i),
\end{align*}
for $x>0$ and
\[\P(\tilde{\Phi}_{i+1}=-\infty\,|\, \tilde{\Phi}_i=-\infty)=1.\]
Thus the discrete generator of $(\tilde{\Phi}_i)_{i\in\N_0}$ on the time scale $i=\varrho_N^{-2}\mu_N^{-1}t$ is given by

\begin{align*}
A_Ng(n):=&\frac{1}{\varrho^2_N\mu_N}\E\big[ g(\tilde{\Phi}_{i+1})-g(x)\,\big|\, \tilde{\Phi}_i=x\big]\\
=& \frac{1}{\varrho^2_N\mu_N}\Big(\frac{C(\gamma)\mu_N\varrho_N}{xr_0}\P(D_{i+1}\,|\, D_i)(g(x+\frac{\varrho_N}{r_0})-g(x))+\P(D^c_{i+1}\,|\, D_i)(g(-\infty)-g(x))\Big)\\
=&\frac{C(\gamma)}{\varrho_N r_0x}\P(D_{i+1}\,|\, D_i)(g(x+\frac{\varrho_N}{r_0})-g(x))+\frac{\P(D^c_{i+1}\,|\, D_i)}{\varrho_N^2\mu_N}(g(-\infty)-g(x)).
\end{align*}

Due to \eqref{eq:noclon_day} and Assumption A iii), the r.h.s. converges for a continuously differentiable function $g:\R\to\R$ that vanishes at $\infty$,  as $N\to\infty$ to 
\[Ag(x):= \frac{C(\gamma)}{r_0^2x}g'(x)\]as $N\to \infty$ (as can be seen from Taylor's expansion, compare the proof of Lemma \ref{lem:fix_epsilon}). This, in turn, is the generator of the solution to the (deterministic) differential equation
\[\dot{h}(t)=\frac{1}{h(t)}\frac{C(\gamma)}{r_0^2}, \,\,t\geq 0,\]
whose solution (for the initial value $h(0)=1$) is $f.$
So we can apply Theorem 4.2.6 in \cite{ethierkurtz} to conclude that $\tilde{\Phi}$ and then $\Phi$ converges in distribution to $(f(t))_{ t\geq 0}$ in the Skorokhod topology. Convergence of $F$ follows from the relation \eqref{eq:FPhi}. Since $f$ is continuous, this amounts to locally uniform convergence in distribution.
\end{proof}

\begin{proof}[Proof of Corollary \ref{Epistasis}.]
The proof is as for Theorem \ref{thm:adapt}, with the only difference that now we replace $\tilde{\Phi}$ by $\tilde{\Phi}^\psi,$ with transition probabilities for $x\geq 1$
\begin{align*}
\P(\tilde{\Phi}^\psi_{i+1}=x+\frac{\psi(x)\varrho_N}{r_0}\,|\, \tilde{\Phi}^\psi_i=x)&=\frac{C(\gamma)\mu_N\varrho_N\psi(x)}{xr_0}\P(D_{i+1}\,|\, D_i),\\
\P(\tilde{\Phi}^\psi_{i+1}=x\,|\, \tilde{\Phi}^\psi_{i}=x)&=\Big(1-\frac{C(\gamma)\mu_N\varrho_N\psi(x)}{xr_0}\Big)\P(D_{i+1}\,|\, D_i),\\
\P(\tilde{\Phi}^\psi_{i+1}=-\infty\,|\, \tilde{\Phi}^\psi_i=x)&=\P(D_{i+1}^c\,|\, D_i),
\end{align*}
for $x>0$ and
\[\P(\tilde{\Phi}^\psi_{i+1}=-\infty\,|\, \tilde{\Phi}^\psi_i=-\infty)=1.\]
which leads to a slightly different discrete generator
\begin{eqnarray*}
A_N^\psi g(x)&=&\frac{C(\gamma)\psi(x)}{\varrho_N r_0x}\P(D_{i+1}\,|\, D_i)\big(g(x+\frac{\psi(x)\varrho_N}{r_0})-g(x)\big)+\frac{\P(D^c_{i+1}\,|\, D_i)}{\varrho_N^2\mu_N}(g(-\infty)-g(x)).\\
\end{eqnarray*}
Thus we get 
\[
 \lim_{N\to\infty}A^\psi_Ng(x)=\frac{\psi(x)^2C(\gamma)}{r_0^2x}g'(x)
\]and we conclude as above. {In particular, solving 
\[\dot{h}(t)=\frac{C(\gamma)}{r_0^2}\frac{1}{h(t)^{2q+1}}\]
yields \eqref{fitness-epistasis}.}
\end{proof}

\begin{appendix}
\section{Basics on Yule processes and proof of Theorem \ref{thm:kingman}}\label{app:yule}

\begin{definition}[Yule process] \label{defYule}
A Yule process with rate $r$ is a continuous-time Markov process taking values in $\mathbb{N}$ such that the transition rates are given by: 
$$\left\{\begin{aligned}
n\to n+1\,\,\,\,\,& \text{at rate}& rn\\
n\to \,\text{others}\,\,\,\,\, &\text{at rate}&\, 0.
\end{aligned}\right.$$
\end{definition}
\begin{remark}\label{rem:repro}
Consider a population model starting with $n_0$ individuals, where each individual reproduces independently at rate $r$ by splitting into two individuals. Then counting the total number of individuals, one gets a Yule process. This is the population model which we consider in the Lenski experiment during one day, with starting population size $n_0=N$. 
\end{remark}
The next lemma is well-known. For part a) see e.g. \cite{AthreyaNey}, p. 109; part b) is due to the indepencence of the branching. 

\begin{lemma}\label{xgeo} Let $Z^r$ be a Yule process with rate $r$.\\
a) If $Z^r(0)=1$, then, for $t> 0$, $Z^r(t)$ follows a geometric distribution with parameter $e^{-rt}.$ 
\\
b)
 If $Z^r(0)=n_0\in\N,$ then $Z^r(t)$ follows a negative binomial distribution with parameters $n_0$ and $e^{-rt}.$ In particular,
 \[\E[Z^r(t)]=n_0e^{rt}, \quad \mbox{ and }\quad \var(Z^r(t))=e^{rt}(e^{rt}-1)n_0.\]
 \end{lemma}

 The next lemma shows that $\varsigma_N$ is asymptotically equal to $\sigma$.
\begin{lemma}\label{lem:conv_time} Let $\varsigma_N$ and $\sigma=\sigma_0$ be as defined in \eqref{tau} and \eqref{sigmak}. Then
$$\varsigma_N\stackrel{(d)}{\to}\sigma.$$
\end{lemma}
\begin{proof}
During one day in the Lenski experiment, consider the population consisting of $N$ subpopulations each of whose sizes follows an independent Yule process with parameter $r$.  
Let {$Z^r_{N}(t)$} denote the size of total population at time $t$. Then {$Z^r_{N}(t)$} is  the sum of  $N$ i.i.d geometric variables with parameter $e^{-rt}$.  Let $\varepsilon>0$. Then due to the law of large numbers 
$$\mathbb{P}\big(\frac{Z^r_{N}(\sigma-\varepsilon)}{\gamma N}<1\big)\stackrel{N\to\infty}{\to} 1;\,\,\, \mathbb{P}\big(\frac{Z^r_{N}(\sigma+\varepsilon)}{\gamma N}>1\big)\stackrel{N\to\infty}{\to} 1.$$
Therefore $\mathbb{P}(\sigma-\varepsilon\leq \varsigma_N\leq \sigma+\varepsilon)\stackrel{N\to\infty}{\to}1.$ Since $\varepsilon$ can be arbitrarily small, the lemma follows.  
\end{proof}
\textbf{Proof of Theorem \ref{thm:kingman}.}
 This is a direct application of Theorem 2.1 in \cite{Moehle-Sagitov}. Fix a generation in the Cannings model and let $c_N$ be the probability
 for a pair of individuals to be coalesced in the previous generation and $d_N$ the probability for a triple of individuals to be coalesced in the previous generation.
 Then it suffices to prove that 
\begin{equation}\label{cndn}
c_N\stackrel{N\to\infty}{\to}0,\,\,d_N/c_N\stackrel{N\to\infty}{\to}0.
\end{equation} 
Notice that $c_N, d_N$ do not depend on the generation since the reproduction, sampling and labeling in each day do not depend on the past and on the future. Therefore we can consider a typical day (the population at the beginning of a day constitutes a generation) and take the notations at the beginning of Section 2.1.1.  Let $Y_t^i$ be the size of the family of individual $i$ at time $t$. Then 
$$Z_t^N=Y_t^1+Y_t^2+\cdots Y_t^N,$$
with $(Y_t^i)_{1\leq i\leq N}$ identically and independently distributed as a geometric distribution with parameter $e^{-rt}$. The day ends at time
$\sigma=\frac{\log \gamma}{r}$ and notice that the population for the next day will be chosen uniformly, hence one can express $c_N, d_N$ as follows:
$$c_N=\mathbb{E}\big[\frac{\sum_{i=1}^N{Y_{\sigma}^i\choose 2}}{{Z_{\sigma}^N\choose 2}}\big]\sim \frac{2(1-\frac{1}{\gamma})}{N},\,\,d_N=\mathbb{E}\big[\frac{\sum_{i=1}^N{Y_{\sigma}^i\choose 3}}{{Z_{\sigma}^N\choose 3}}\big]=O(N^{-2}),$$
 which gives \eqref{cndn}, and thus completes the proof. \hfill $\Box$

\section{Properties of near-critical Galton-Watson processes}\label{app:GW}

The following lemma (Theorem 3 of \cite{Athreya}, and see also Theorem 5.5 in \cite{HJV} under weaker conditions) provides the survival probability for certain near-critical Galton-Watson trees. 

\begin{lemma}\label{lem:GW}
Consider a sequence of  supercritical Galton-Watson processes $(G^N_i)_{i\in\N_0}$, $N=1,2,\ldots$, with offspring mean $1+\beta_N$ (with $\beta_N\to 0$) and offspring variance $\sigma^2+v_N$ (with $v_N\to 0$) and uniformly bounded third moment, starting from one ancestor
in generation 0. Then the survival probability $\phi_N$ obeys $\phi_N \sim \frac {2\beta_N}{\sigma^2}.$ 
\end{lemma}

\begin{lemma}\label{Prop:criticalGW}
Let $(G^N_i)_{i\in\N_0}, N=1,2,...$ be as in Lemma \ref{lem:GW}. Assume that $\beta_NN \to \infty$ as $N\to \infty$.  Then, for every $\varepsilon>0$, $\P(\exists i: G^N_i\geq \varepsilon N)\sim \P(\lim_{i\rightarrow\infty}G^N_i=\infty)$.
\end{lemma}
\begin{proof} Again let $\phi_N$ be the survival probability of $G^N$ started in one individual. Then
$$\P(\lim_{i\to\infty}G_i=\infty|\exists i:G_i\geq \varepsilon N)\geq 1-(1-\phi_N)^{\varepsilon N}\sim 1-(1-\frac {2\beta_N}{\sigma^2})^{\varepsilon N}\to 1,\, N\to\infty.$$
\end{proof}

\begin{lemma}\label{lem:timefixGW}
Let $(G^N_i)_{i\in\N_0}, N=1,2,...$ be as in Lemma \ref{lem:GW}. Assume  that $\beta_N \sim c N^{-b}, N=1,2,\ldots$, for some $c > 0$ and $b \in (0,1)$. For fixed $\varepsilon \in (0,1)$,  let $\omega_N
:= \inf\{i\ge 0: G^N_i \ge \varepsilon N\}$. Then we have for any $\delta>0$
\[\lim_{N\to\infty}\P_1(\omega_N
> \beta_N^{-1-\delta}\,|\, \omega_N
 < \infty)=0.\]
Further, let $\upsilon_N:=\inf\{i\geq 0: G_i^N=0\}.$ Then for any $\delta>0,$ for $N$ large enough,
\be \label{timeextGW}\P_1(\upsilon_N
> \beta_N^{-1-\delta}\,|\, \upsilon_N
 < \infty)\leq e^{-N^{b\delta}}.\ee
\end{lemma}
\begin{proof}
First we consider the difference between conditioning $G^{N}$ on survival (forever) and on reaching $\varepsilon N$, respectively. Since we know (from Lemma B1) that 
\begin{equation}\label{surv}
\mathbb P_1 (G^N \mbox{ survives}) \sim \frac {2\beta_N}{\sigma^2} \sim c' N^{-b},
\end{equation}
we can infer, using the strong Markov property, that
$$\mathbb P_1 (G^N \mbox{ reaches } \varepsilon N \mbox{ and } G^N \mbox{ does not survive })\le \mathbb P_{\lfloor \varepsilon N\rfloor} (G^N \mbox{ does not survive }) $$
\begin{equation}\label{expest}
=\left(1-\phi_N\right)^{\lfloor \varepsilon N\rfloor }\le \left(1-c'N^{-b}\right)^{\lfloor \varepsilon N\rfloor } \le \exp(-c(\varepsilon)  N^{1-b}).
\end{equation}
Thus we can estimate
\begin{equation*}
\mathbb P_1 (\omega_N
 > \beta_N^{-1-\delta} | G^N \mbox{ reaches } \varepsilon N) = \frac 1{\mathbb P_1(G_N \mbox{ reaches } \varepsilon N)}\mathbb   P_1(\omega_N
 > \beta_N^{-1-\delta},  G^N \mbox{ reaches } \varepsilon N)
\end{equation*}
$$\le \frac 1{\mathbb P_1(G^N \mbox{ reaches } \varepsilon N)}\mathbb P_1( G^N\mbox{ reaches } \varepsilon N \mbox{ and does not survive}) $$
$$   + \frac 1{\mathbb P_1(G^N \mbox{ survives })} \mathbb   P_1(\omega_N
 > \beta_N^{-1-\delta},  G^N \mbox{ survives}).$$
The first summand on the r.h.s tends to $0$ as $N\to \infty$ because of \eqref{surv} and \eqref{expest}. Thus, for proving the lemma it suffices to show that 
\begin{equation}\label{suff}\lim_{N\to\infty}\P_1(\omega_N
> \beta_N^{-1-\delta}| G^N \mbox{ survives })=0.
\end{equation}
Let $\phi_N$ be the survival probability of $G^N$, and denote by $H^N_i$, $i=0,1,\ldots$, the generation sizes of those individuals that have an infinite line of descent, conditioned on survival of $G^N$.  Then we have (cf. Proposition 5.28 in \cite{LP})
\begin{equation*}\label{h1n} f^\ast(s):= \sum_{k\ge 0} s^k\P_1(H_1^N=k) = \E_1[s^{H_1^N}]=\frac{\E[(1-\phi_N+\phi_Ns)^{G_1^N}]-(1-\phi_N)}{\phi_N}, s\geq 0.\
\end{equation*}
Obviously, $\P_1(H_1^N=0) =f^\ast(0) = 0$ and 
$\P_1(H_1^N=1) =(f^\ast)'(0) = \E[G_1^N(1-\phi_N)^{G_1^N-1}]$,
which, using Taylor expansion, is transformed to
\begin{align}
\E[G_1^N\Big(1-(G_1^N-1)\phi_N+\frac{(G_1^N-1)(G_1^N-2)\phi_N^2}{2}(1-t\phi_N)^{G_1^N-3}\Big)\nonumber\\  
=\E_1[G_1^N(1-(G_1^N-1)\phi_N)]+O(\phi_N^2)
= 1-\beta_N+o(\beta_N),\label{eq:conditionedGW}
\end{align}
where $t= t(G_1^N) \in (0,1)$. The first equality is due to the assumption in Lemma \ref{lem:GW} that the third order moment of $G_1^N$ is uniformly bounded.
We can thus infer that, for any fixed $\eta \in (0,1)$,
\begin{equation*}\P_1(H_1^{N}\geq 2)\geq \eta\beta_N, \text{when $N$ is large enough}.
\end{equation*} 
We can now give a lower bound for  $G^N_i$, conditioned on survival of $G^N$,  in two steps: first by $H^N_i$, and then by a (discrete time) Galton-Watson process  with offspring distribution $(1-\eta\beta_N)\delta_1 + \eta\beta_N \delta_2$. Call this process $B^{N}$.  With $\frac 1{\lfloor \eta \beta_N \rfloor}$ generations as a new time unit, the sequence of processes $B^N$ converges, as $N \to \infty$, to a standard Yule process. This means that, for every fixed $t > 0$, at a time of $\lfloor t   \eta\beta_N \rfloor^{-1}$ generations, $B^N$ has an approximate geometric distribution with parameter $e^{-t}$. Thus we conclude after  $\lfloor  \beta_N \rfloor^{-1-\delta}$ generations, $B^N$ (and a fortiori also $G^N$ when conditioned to survival) is larger than $\varepsilon N$ with probability tending to $1$ as $N\to \infty$. This shows \eqref{suff}, and concludes the proof of the first statement. \\
 For the last statement, observe that by Theorem 5.28 of \cite{LP} the distribution of $(G^N_i)$ conditioned on extinction is equal to the distribution of a Galton Watson process with probability generating function 
\begin{equation*} \overline{f}(s):= (1-\phi_N)^{-1}\sum_{k\ge 0} ((1-\phi_N)s)^k\P_1(G_1^N=k).\
\end{equation*}
Thus we have $$\E_1[G_1^N|G^N\mbox{ dies out}]=\overline{f}^\prime(1)=\E[G_1^N(1-\phi_N)^{G_1^N-1}]=1-\beta_N+o(\beta_N),$$ where the last equality follows from equation \eqref{eq:conditionedGW}. 
Then, by Proposition 5.2 in \cite{LP} we observe that 
\be \label{eq:exp_dieout}\E_1[G_{\lceil \beta_N^{-1-\delta}\rceil}^N|G^N\mbox{ dies out}]=(1-\beta_N+o(\beta_N))^{\beta_N^{-1-\delta}}\leq e^{-N^{b\delta}}\ee
so we conclude
$$\P_1(v_N>\beta_N^{-1-\delta}|v_N<\infty )=\P_1(G_{\lceil \beta_N^{-1-\delta}\rceil}^N>0|G^N\mbox{ dies out})\leq \E_1[G_{\lceil \beta_N^{-1-\delta}\rceil}^N|G^N\mbox{ dies out}]\leq e^{-N^{b\delta}}.$$
\end{proof}
\end{appendix}

\paragraph{\bf Acknowledgements} We thank Ellen Baake, Jochen Blath, Michael Desai, Thomas Hindr\'e, Tom Kurtz and Todd Parsons for stimulating discussions, Thomas Lenormand and Sylvie M\'eleard for pointing us to reference \cite{Chevin}, Valdimir Vatutin for the hint to references \cite{Athreya, HJV},  and T. Hindr\'e and T. Lenormand for illuminating lectures on experimental evolution during the 2013 Spring School in Aussois. This work was supported in part  by the priority program SPP 1590 ``probabilistic structures in evolution'' of the German Science Foundation (DFG). AGC acknowledges support by DFG RTG 1845 ``Stochastic Analysis and Applications in Biology, Finance and Physics'', the Berlin Mathematical School (BMS), and the Mexican Council of Science (CONACyT) in collaboration with the German Academic Exchange Service (DAAD). LY acknowledges support by La Fondation Sciences Math\'ematiques de Paris and the Swedish Research Council through grant no. 2013-4688.
\bibliographystyle{plain}
\bibliography{bibliography}
\end{document}